\documentclass[12pt]{amsart}
\usepackage[a4paper,margin=2.5cm]{geometry}
\usepackage{amsmath,amsthm,amssymb,amscd,amsfonts}
\usepackage{tikz,blkarray}
\usetikzlibrary{positioning}
\usepackage{cleveref}

\Crefname{secinapp}{Appendix}{Appendices}
\crefname{enumi}{Procedure}{parts}

\theoremstyle{plain}
\newtheorem{theorem}{Theorem}[section]
\newtheorem{proposition}[theorem]{Proposition}
\newtheorem{lemma}[theorem]{Lemma}

\numberwithin{equation}{section}

\theoremstyle{definition}
\newtheorem{definition}[theorem]{Definition}
\newtheorem{remark}[theorem]{Remark}
\newtheorem{example}[theorem]{Example}
\newtheorem{algorithm}[theorem]{Algorithm}

\newcommand{\cP}{\mathcal{P}}
\newcommand{\cR}{\mathcal{R}}

\newcommand{\cM}{\mathcal{M}}
\newcommand{\cT}{\mathcal{T}}
\newcommand{\cO}{\mathcal{O}}
\newcommand{\cV}{\mathcal{V}}
\newcommand{\cE}{\mathcal{E}}
\newcommand{\cS}{\mathcal{S}}

\newcommand{\Z}{\mathbb{Z}}

\newcommand{\R}{\mathbb{R}}

\newcommand{\fS}{\mathfrak{S}}

\DeclareMathOperator{\CM}{CM}
\DeclareMathOperator{\DCM}{\overline{CM}}
\DeclareMathOperator{\Pic}{Pic}
\DeclareMathOperator{\im}{im }
\DeclareMathOperator{\GL}{GL}
\DeclareMathOperator{\St}{St}
\DeclareMathOperator{\Lk}{Lk}
\DeclareMathOperator{\wed}{wed}
\DeclareMathOperator{\Proj}{Proj}
\DeclareMathOperator{\V}{Vert}
\DeclareMathOperator{\co}{co}
\DeclareMathOperator{\len}{len}

\title[]{An algorithmic strategy for finding characteristic maps over wedged simplicial complexes}
\author{Suyoung Choi}
\address{Department of mathematics, Ajou University, 206, World cup-ro, Yeongtong-gu, Suwon 16499, Republic of Korea}
\email{schoi@ajou.ac.kr}
\author{Mathieu Vall\'ee}
\address{Univ Rennes, CNRS, IRMAR - UMR 6625, F-35000 Rennes, France}
\email{mathieu.vallee@ens-rennes.fr}

\date{\today}

\thanks{This project was supported by the National Research Foundation of Korea Grant funded by the Korean Government (NRF-2019R1A2C2010989).}

\begin{document}
\begin{abstract}
The puzzle method was introduced by Choi and Park as an effective method for finding non-singular characteristic maps over wedged simplicial complexes $K(J)$ obtained from a given simplicial complex $K$. We study further the mod $2$ case of the puzzle method. We firstly describe it completely in terms of linear algebraic language which allows us to develop a constructive puzzle algorithm. We also analyze our algorithm and compare its performances with other known algorithms including the Garrison and Scott algorithm.
\end{abstract}
\maketitle

\section{Introduction}
Let $K$ be a pure simplicial complex of dimension $n-1$ on the vertex set $V$, where $n$ is a positive integer.
A map $\lambda \colon V \to \Z_2^n$ is called a (mod $2$) \emph{characteristic map} over $K$ if it satisfies the following \emph{non-singularity condition}:
if $\{i_1, \ldots, i_n\} \in K$, then $\{\lambda(i_1), \ldots, \lambda(i_n)\}$ is linearly independent.
Two characteristic maps over $K$ are said to be \emph{DJ-equivalent} if one can be obtained from the other by changing basis of $\Z_2^n$.
Denote by $\CM(K)$ the DJ-equivalence classes of characteristic maps over $K$.
It is known that $\CM(K)$ provides the classification of some important classes of manifold.
For instance, if $K$ is polytopal, $\CM(K)$ is the classification of small covers introduced in \cite{Davis-Januszkiewicz1991} whose orbit space is combinatorially isomorphic to $K$.

One important question in toric topology is ``how to find $\CM(K)$ for a given $K$?".
There have been many attempts to list up all mod $2$ characteristic maps over some specific (polytopal) PL spheres such as a dodecahedron
\cite{Garrison-Scott2003}, prisms \cite{Cai-Cehn-Lu2007}, cubes \cite{Choi2008}, and polytopes with a few vertices \cite{Erokhovets2011}, \cite{Choi-Park2016}.

For this purpose, one remarkable algorithm was introduced by Garrison and Scott \cite{Garrison-Scott2003}.
This algorithm produces $\CM(K)$ for any pure simplicial complex $K$.
However, since it is a branch-and-bound algorithm, its complexity becomes relatively bad when increasing the dimension $n-1$ or the number of vertices $m$, and, hence, it can be used for neither high dimensional cases nor for equipping many vertices.

Another remarkable algorithm was provided after by Choi and Park (\cite{Choi-Park2016}, \cite{CP_wedge_2}).
In order to present their method, we need to introduce some notations.
The \emph{Picard number} $\Pic(K)$ of $K$ is defined as $\Pic(K)=m-n$, where $m = |V|$.
For $v \in V$, the \emph{wedge} of $K$ at $v$ is the simplicial complex on $V\cup\{v_1,v_2\} \setminus  \{v\}$ defined by
	\begin{equation}\label{Wedge_def}
	   \wed_v(K) := (I \ast \Lk_K(v))\cup (\partial I \ast K \setminus \{v\}),
	\end{equation}
	where $I$ is the $1$-simplex with vertices $\{v_1,v_2\}$, $K \setminus F:= \{\sigma \in K \mid F \not\subset \sigma\}$, for a face $F\in K$, and $\Lk_K(v)$ is the link of $K$ at the vertex $v$.
Set $V = [m] := \{1,\ldots,m\}$. More generally, for an $m$-tuple $J=(j_1, \ldots, j_m)$ of positive integers, denote by $K(J)$ the \emph{wedged simplicial complex} obtained from $K$ after performing $j_v-1$ wedges on each vertex $v \in V$.
It should be noted that a wedge operation preserves the Picard number.
If $K$ is not obtained from a lower dimensional simplicial complex by a wedge operation, then $K$ is called a \emph{seed}.
Any simplicial complex $K$ which is not a seed can always be represented as a wedged simplicial complex $L(J)$ with $L$ being a seed.

Choi and Park's method relies on obtaining $\CM(\wed_v (K))$ from $\CM(K)$.
This process can be interpreted as putting stones on a certain board following specific rules, accordingly this method is called \emph{the puzzle method}.
A puzzle is said to be \emph{realizable} if all stones of the board respect the rules.
For a seed $L$, we firstly prepare $\CM(L)$. Its low dimension makes the Garrison and Scott algorithm usable.
For a wedged simplicial complex $K = L(J)$ and $J=(j_1, \ldots, j_m)$, one can construct $\CM(K)$ from $\CM(L)$ by applying a puzzle method inductively on $j_1 + \cdots + j_m$.
In many ways this method is quite powerful since it can handle plenty of interesting cases such as wedges of polygons \cite{Choi-Park2019}.

Considering an algorithmic viewpoint, the puzzle method would accelerate the process for constructing $\CM(L(J))$ for $L$ a seed and would be faster than computing it directly by the Garrison and Scott algorithm. In addition, it was shown in \cite{CP_wedge_2} that, for a fixed Picard number $p$, the number of seeds $L$ with $\Pic(L)=p$ satisfying $\CM(L) \neq \emptyset$ is finite.
More precisely, $n \leq 2^p-p$.
In other words, when $p$ is small and the dimension is high, then a PL sphere $K$ is either a wedged PL sphere or it does not support any mod $2$ characteristic map.
Thus, for a small Picard number, an algorithm based on the puzzle method would be a fine upgrade of the Garrison and Scott algorithm.
As a by-product, if one can find all seeds of fixed Picard number $p$ that support a mod $2$ characteristic map, an algorithm based on the puzzle method enables us to classify some important classes of closed manifolds including small covers and real toric manifolds whose first $\Z_2$-cohomology groups are of rank $p$.

A straightforward algorithm based on the puzzle method would be to enumerate all possible stone configuration on the board and check if the rules introduced in \cite{CP_wedge_2} are respected in each such configuration.

However, the latter rules are somehow complicated and require a lot of preliminary computation. The first rule demands all edges of the board to correspond to a mod $2$ characteristic map over $\wed_v(L)$, and this requires to compute $\CM(\wed_v(L))$ for each vertex $v$ of $L$, i.e. $m$ times. The second one is that every subsquares of the board should correspond to a mod $2$ characteristic map over a simplicial complex obtained from $L$ after two consecutive wedge operations, namely $\wed_{v,w}(L)$, and this requires to compute $\CM(\wed_{v,w}(L))$ for each unordered pair of distinct vertices $(v,w)$ of $L$, i.e. $m(m-1)/2$ times.

Both of these rule computations would be time consuming in any algorithm based on the puzzle method. Moreover, our intuition makes us believe that we could try to fill the board with stones in a constructive way by respecting the rules step by step. Consequently, until now, this method seemed promising but have not been analyzed because of the complexity of these rules.

In this paper, we zero in on finding a strategy for the puzzle method to be usable as an algorithm.
We describe the puzzle method completely in terms of linear algebraic language. It provides explicit formulae for all possible mod $2$ characteristic maps over a wedged seed and over a twice wedged seed. These formulae only require the computation of $\CM(L)$ for the seed $L$ and thus make the previously mentioned time consuming computations unnecessary.

Furthermore, we focus on the key concept that a realizable puzzle only requires a few stones positioned on the board to have all of its other stone positions determined \cite{CP_wedge_2}. From this we transform the greedy ``enumerating all cases'' puzzle algorithm onto an elegant constructive and procedural algorithm which uses the edge and square rules previously computed for building a realizable puzzle.

We compute its complexity, and show that it is much faster than a direct use of the Garrison and Scott algorithm on $L(J)$ when we tend to increase the dimension $n-1$.

\section{Dual characteristic maps}
Let $K$ be a PL sphere on $[m]=\{1, \ldots, m\}$ with $\dim(K) = n-1$. A non-singular characteristic map $\lambda$ over $K$ can be associated to an element of a matrix group $\lambda = \left(\begin{array}{c|c|c|c}
\lambda(1) & \lambda(2) & \cdots & \lambda(m)
\end{array}\right)\in \cM(n,m,\Z_2)$, where $\cM(n,m,\Z_2)$ is the set of $n\times m$ $\Z_2$-matrices.
The order in which the columns appear does not matter since we can relabel the vertices of $K$. The only important data is the combinatorial data stored by $\lambda$, namely the fact it satisfies the non-singularity condition. Note that the map $\lambda$ can also be interpreted as a \emph{coloring} of the vertices with elements of $\Z_2^n$ respecting a coloring rule which is the non-singularity condition. Insofar as, vertices being part of the same face must have linearly independent colors.

Now, we denote the set of the characteristic maps over $K$ by $\Lambda(K)$ and consider the $\GL(n,\Z_2)$-action of left multiplication on $\Lambda(K)$, where $\GL(n, \Z_2)$ is the general linear group of degree $n$ over $\Z_2$.
One can check that $g \cdot \lambda$ for any $g\in \GL(n,\Z_2)$ still satisfies the non-singularity condition for $K$.
Hence, $\GL(n,\Z_2)$ acts on $\Lambda(K)$ and we will call an orbit of $\GL(n,\Z_2) \curvearrowright \Lambda(K)$ a \emph{DJ class} (of characteristic maps) over $K$.
We denote by $\CM(K) = \GL(n,\Z_2)\setminus \Lambda(K)$.

Since the labels of the vertices of $K$ can be isomorphically modified with an element of a permutation group $\fS_{[m]}$ we can always consider that the face $\{1,\ldots,n\}$ is a maximal face of $K$.
Thus by processing the Gaussian elimination algorithm on an element $\lambda\in\CM(K)$, we can obtain a representative of the DJ class being $
    \lambda = \left(\begin{array}{c|c}
        I_n & M
    \end{array}\right)
\in \cM(n,m,\Z_2)$, where $I_n$ is the identity matrix of size $n$, and $M$ is an $n \times (m-n)$ matrix.

We now construct a new object which will store the same combinatorial data than $\lambda$ as follows. We see $\lambda\in\CM(K)$ as a linear map $\lambda \colon \Z_2^m \to \Z_2^n$ using its matrix representation.
Since $\lambda$ has rank $n$, by the rank theorem, $\dim(\ker(\lambda)) = m-n$.
We take a basis $\left\{v_1,\ldots,v_{m-n} \right\}$ of $\ker(\lambda)$.
Let us define $\lambda^\star := \left(\begin{array}{c|c|c} v_1 & \cdots & v_{m-n}\end{array}\right)\in \cM(m,m-n,\Z_2)$.
We have $\lambda\lambda^\star = 0$.
The map $\lambda^\star$ is well-defined since $(g \cdot \lambda)\lambda^{\star} = g \cdot(\lambda\lambda^\star) = 0$ for all $g \in \GL(n,\Z_2)$.
We denote by $\bar{\lambda}(1),\ldots,\bar{\lambda}(m)$ the rows of $\lambda^\star$.
Then, the corresponding map $\bar{\lambda} \colon [m]\to\Z_2^{m-n}, i \mapsto \bar{\lambda}(i)^t$ is called the \emph{dual characteristic map}  associated to $\lambda$.

Furthermore, one can see that $\lambda^\ast$ is independent from the choice of a basis for the kernel. For all $h\in \GL(m-n, \Z_2)$, $\lambda(\lambda^\star \cdot h) = (\lambda\lambda^\star) \cdot h = 0$.
We will then call DJ class of the dual characteristic map $\bar{\lambda}$ the orbit $\DCM(K):=\{\lambda^{\star} \cdot h \colon h\in \GL(m-n, \Z_2)\}$.

\begin{remark}
    There is a one-to-one correspondence between $\CM(K)$ and $\DCM(K)$.
    One can notice that if we choose a representative $\lambda = \left(\begin{array}{c|c} I_n & M \end{array}\right)$ (the matrix $M$ is unique for each DJ class) of a DJ class, then we can map this representative to $\bar{\lambda}^t = \left(\begin{array}{c}M \\ \hline I_{m-n}\end{array}\right)$ and we see that $\lambda \bar{\lambda}^t = \left(\begin{array}{c}M+M\end{array}\right) = 0$.
Conversely, from a dual DJ class $\bar{\lambda}$ we can find a representative of the form $\bar{\lambda}^t = \left(\begin{array}{c}M \\ \hline I_{m-n}\end{array}\right)$ and $M$ is unique for each class.
\end{remark}

In addition to this correspondence, it was observed in \cite[Corollary~7.33]{BP2002} that each DJ class $\lambda$ and its associated dual $\bar{\lambda}$ share the same combinatorial data and this shows up as a \emph{dual non-singularity condition} for the dual DJ classes.

\begin{proposition}
Let $K$ be a PL sphere on $[m]$ of dimension $n-1$, $\lambda$ a characteristic map over $K$, and $\bar{\lambda}$ its dual.
Let $S$ be a subset of $[m]$.
The following are equivalent:
\begin{enumerate}
\item $\bar{\lambda}(S^c)$ is a basis of $\Z_2^{m-n}$ (dual non-singularity condition);
\item $\lambda(S)$ is a basis of $\Z_2^{n}$ (non-singularity condition).
\end{enumerate}
\end{proposition}
\begin{proof}
Let us suppose that $\bar{\lambda}(S^c)$ is a basis of $\Z_2^{m-n}$. We write $\lambda = \left(\begin{array}{c|c} \lambda_{|S} & \lambda_{|S^c}\end{array}\right)$ and $\lambda^\star = \left(\begin{array}{c}
\lambda^\star_{|S} \\ \hline  \lambda^\star_{|S^c}\end{array}\right)$ which are block matrices representing $\lambda$ and $\bar{\lambda}$ respectively, we have
$$\lambda \lambda^\star = \lambda_{|S}\lambda^\star_{|S} + \lambda_{|S^c}\lambda^\star_{|S^c} = 0,$$
thus
$$\im \lambda_{|S} \supset \im \lambda_{|S}\lambda^\star_{|S} = \im \lambda_{|S^c}\underset{\text{is of full rank}}{\underbrace{\lambda^\star_{|S^c}}} = \im \lambda_{|S^c},$$
and then
$$\im \lambda = \im \lambda_{|S},$$
because we would have a problem with the rank of $\lambda$ otherwise.
Hence, the rank of $\lambda_{|S}$ is $n$.
The converse is proved very similarly by using the co-images instead of the images.
\end{proof}

\begin{remark}One can notice that the proof only requires $K$ to be a simplicial complex. However for the dual characteristic map to be interesting and usable, we will need at least $K$ to be pure, and in our case of study, to be a PL sphere.
\end{remark}

Now we know that both a characteristic map and its dual over $K$ store the same combinatorial data. In addition, there is a one-to-one relation between each DJ class and its dual.
From this, in all that follows we will not hesitate to mistake the notion of characteristic map with the one of dual characteristic map and use one or another depending on which one is the most convenient for the context.

\section{The puzzle method for finding characteristic maps over wedged seeds}\label{PuzzleCP}

The goal of this section is to introduce the puzzle method for finding $\CM(K)$, with $K$ a PL sphere, based on \cite{CP_wedge_2}.
It will be necessary for the next section which will describe a drastic improvement for the mod $2$ characteristic maps case.
It should be noted that the original method of Choi and Park is applicable to not only mod $2$ characteristic maps but also its entire version $\lambda \colon V \to \Z^n$.
However, throughout this paper, we will only deal with the case of mod $2$ characteristic maps.

The following definition is an inductive generalization of the wedge operation defined in the introduction.

\begin{definition}[\cite{BBCG2015}]
    Let $K$ be a simplicial complex on $[m]$, and $J = (j_1,\ldots,j_m) \in (\Z_{>0})^m$.
    We define the wedged simplicial complex $K(J)$ as the simplicial complex obtained after performing $j_v-1$ wedges on each vertex $v\in[m]$.
    We define $\len(J):= \left(\sum_{k=1}^m j_k\right) -m$ to be the total number of wedges performed with $J$.
\end{definition}
\begin{remark}Informally it is the simplicial complex
    $$K(J) = \underbrace{\wed_1\ldots\wed_1}_{j_1 \text{ times}}\ldots\underbrace{\wed_v\ldots\wed_v}_{j_v \text{ times}}\ldots\underbrace{\wed_m\ldots\wed_m}_{j_m \text{ times}}(L),$$
    but writing it in this way is not convenient and requires some work on the notation.
    
    Thus, this definition provides a compact notation $K=L(J)$, with $L$ being \emph{the seed of $K$} (unique up to isomorphism) and $J$ an $m$ tuple of positive integer representing the wedge operations performed on $L$ to get $K$. Notice that in this case, the $m$-tuple $J$ is not unique, in fact the symmetries of $L$ can lead to $L(J)=L(J^\prime)$ with $J\neq J^\prime$. In addition, one can see that when we write $L(J)$, $L$ does not necessarily need to be a seed. However in all that follows, we will always use this notation when $L$ is the seed of some simplicial complex $K$.
\end{remark}

We can now rephrase our goal. If $L$ is a seed PL sphere on $[m]$ of dimension $n-1$ and $J = (j_1,\ldots,j_m)$ is an $m$-tuple of positive integers, we want to find an algorithm for constructing all  DJ classes over the wedged PL sphere $K=L(J)$.

The operation which is converse to the wedge operation is the \emph{link} operation.

\begin{definition}[Link of a face]
Let $K$ be a simplicial complex on $[m]$ and $\sigma\in K$ a face of $K$. Then the Link of $\sigma$ in $K$ is:
$$\Lk_K(\sigma) := \{\tau\in K \mid \sigma\cup \tau \in K,\tau\cap \sigma = \emptyset \}.$$
\end{definition}

It should be noted that if $K$ is a PL sphere, then the link $\Lk_K(\sigma)$ of any face $\sigma$ of $K$ is also a PL sphere.
Furthermore, if $K$ supports a mod $2$ characteristic map, so do its links.

\begin{definition} \cite{Choi-Park2016}
Let $\lambda$ be a mod $2$ characteristic map over $K$ and $\sigma\in K$ a face of $K$.
The \emph{projection of $\lambda$ with respect to $\sigma$} is
$$(\Proj_\sigma\lambda)(s) = [\lambda(s)] \in \Z_2^n/\langle \lambda(v),v\in\sigma\rangle \simeq \Z_2^{n-|\sigma|},$$
for $s$ a vertex of $\Lk_K(\sigma)$.
Then, $\Proj_\sigma\lambda$ is well-defined and is in $\CM(\Lk_K(\sigma))$.
When $\sigma = \{p\}$ is a vertex, we will simply write $\Proj_\sigma\lambda = \Proj_p\lambda$.
When $\sigma = \{p, q\}$ is a pair of distinct vertices, we also simply write $\Proj_\sigma\lambda = \Proj_{p,q}\lambda$.
\end{definition}

\begin{example}\label{ex_pentagon}Let us illustrate what we meant by ``the operation converse to the wedge is the link operation". Let $\cP_5$ be the pentagon on $\{1,\ldots,5\}$. The wedge of $\cP_5$ at $1$ is illustrated as below.
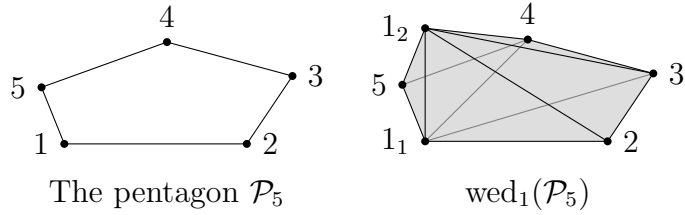
\begin{center}
\begin{tabular}{cc}
\begin{tikzpicture}[scale=3]
\node[circle,fill=black,inner sep=0pt,minimum size=3pt,label=left:{$1$}] at (0,0) {};
\node[circle,fill=black,inner sep=0pt,minimum size=3pt,label=right:{$2$}] at (0.8,0) {};
\node[circle,fill=black,inner sep=0pt,minimum size=3pt,label=right:{$3$}] at (1.0,0.3) {};
\node[circle,fill=black,inner sep=0pt,minimum size=3pt,label=above:{$4$}] at (0.45,0.45) {};
\node[circle,fill=black,inner sep=0pt,minimum size=3pt,label=left:{$5$}] at (-0.1,0.25) {};
\draw (0,0) -- (0.8,0) -- (1.0,0.3) -- (0.45,0.45) -- (-0.1,0.25) -- cycle ;
\end{tikzpicture}
&
\begin{tikzpicture}[scale=3]
\draw (-0.1,0.25) -- (0.45,0.45);
\draw (0,0) -- (1.0,0.3) ;
\draw (0,0) -- (0.45,0.45);
\fill[fill = black!20!white,opacity = 0.4] (0,0) -- (0.8,0) -- (1.0,0.3) -- cycle;
\fill[fill = black!20!white,opacity = 0.4] (0,0) -- (1.0,0.3) -- (0.45,0.45) -- cycle;
\fill[fill = black!20!white,opacity = 0.4] (0,0) --  (0.45,0.45) -- (-0.1,0.25) -- cycle;
\fill[fill = black!20!white,opacity = 0.4] (0,0.5) -- (0.8,0) -- (1.0,0.3) -- cycle;
\fill[fill = black!20!white,opacity = 0.4] (0,0.5) -- (1.0,0.3) -- (0.45,0.45) -- cycle;
\fill[fill = black!20!white,opacity = 0.4] (0,0.5) -- (0.45,0.45) -- (-0.1,0.25) -- cycle;
\fill[fill = black!20!white,opacity = 0.4] (0,0.5) -- (0.8,0) -- (0,0) -- cycle;
\fill[fill = black!20!white,opacity = 0.4] (0,0.5) -- (0,0) -- (-0.1,0.25) -- cycle;
\draw (0,0) -- (0,0.5);
\draw (0,0.5) -- (0.8,0);
\draw (0,0.5) -- (-0.1,0.25);
\draw (0,0.5) -- (1.0,0.3);
\draw (0,0.5) -- (0.45,0.45);
\draw (-0.1,0.25) -- (0,0) -- (0.8,0) -- (1.0,0.3) -- (0.45,0.45);
\node[circle,fill=black,inner sep=0pt,minimum size=3pt,label=left:{$1_1$}] at (0,0) {};
\node[circle,fill=black,inner sep=0pt,minimum size=3pt,label=left:{$1_2$}] at (0,0.5) {};
\node[circle,fill=black,inner sep=0pt,minimum size=3pt,label=right:{$2$}] at (0.8,0) {};
\node[circle,fill=black,inner sep=0pt,minimum size=3pt,label=right:{$3$}] at (1.0,0.3) {};
\node[circle,fill=black,inner sep=0pt,minimum size=3pt,label=above:{$4$}] at (0.45,0.45) {};
\node[circle,fill=black,inner sep=0pt,minimum size=3pt,label=left:{$5$}] at (-0.1,0.25) {};
\end{tikzpicture}
\\
The pentagon $\cP_5$ & $\wed_1(\cP_5)$
\end{tabular}
\end{center}

The link of $\wed_1(\cP_5)$ at vertices $1_1$ and $1_2$ are two copies of $\cP_5$ as follows.
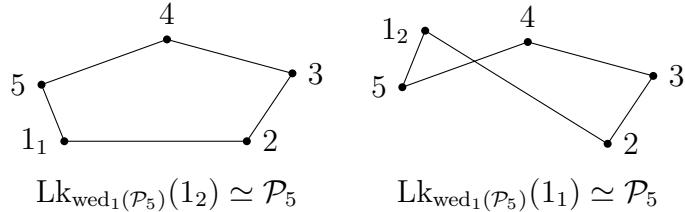
\begin{center}
\begin{tabular}{cc}
\begin{tikzpicture}[scale=3]
\node[circle,fill=black,inner sep=0pt,minimum size=3pt,label=left:{$1_1$}] at (0,0) {};
\node[circle,fill=black,inner sep=0pt,minimum size=3pt,label=right:{$2$}] at (0.8,0) {};
\node[circle,fill=black,inner sep=0pt,minimum size=3pt,label=right:{$3$}] at (1.0,0.3) {};
\node[circle,fill=black,inner sep=0pt,minimum size=3pt,label=above:{$4$}] at (0.45,0.45) {};
\node[circle,fill=black,inner sep=0pt,minimum size=3pt,label=left:{$5$}] at (-0.1,0.25) {};
\draw (0,0) -- (0.8,0) -- (1.0,0.3) -- (0.45,0.45) -- (-0.1,0.25) -- cycle ;
\end{tikzpicture}
&
\begin{tikzpicture}[scale=3]

\draw (-0.1,0.25) -- (0,0.5) -- (0.8,0) -- (1.0,0.3) -- (0.45,0.45) -- cycle;
\node[circle,fill=black,inner sep=0pt,minimum size=3pt,label=left:{$1_2$}] at (0,0.5) {};
\node[circle,fill=black,inner sep=0pt,minimum size=3pt,label=right:{$2$}] at (0.8,0) {};
\node[circle,fill=black,inner sep=0pt,minimum size=3pt,label=right:{$3$}] at (1.0,0.3) {};
\node[circle,fill=black,inner sep=0pt,minimum size=3pt,label=above:{$4$}] at (0.45,0.45) {};
\node[circle,fill=black,inner sep=0pt,minimum size=3pt,label=left:{$5$}] at (-0.1,0.25) {};
\end{tikzpicture}
\\
$\Lk_{\wed_1(\cP_5)}(1_2)\simeq \cP_5$ & $\Lk_{\wed_1(\cP_5)}(1_1)\simeq \cP_5$
\end{tabular}
\end{center}
Let $\Lambda=\begin{blockarray}{cccccc}
1_1 & 1_2 & 2 & 3 & 4 & 5\\
\begin{block}{(ccc|ccc)}
1 & 0 & 0 & 1 & 0 & 1\\
0 & 1 & 0 & 1 & 1 & 0\\
0 & 0 & 1 & 0 & 1 & 1\\
\end{block}\end{blockarray}\in\CM(\wed_1(\cP_5))$. The two projections of $\Lambda$ with respect to the vertices $1_1$ and $1_2$ are
$$\Proj_{1_1}\Lambda=\begin{blockarray}{ccccc}
 1_2 & 2 & 3 & 4 & 5\\
\begin{block}{(cc|ccc)}
1 & 0 & 1 & 1 & 0\\
0 & 1 & 0 & 1 & 1\\
\end{block}\end{blockarray}\text{  and }\Proj_{1_2}\Lambda=\begin{blockarray}{ccccc}
1_1 & 2 & 3 & 4 & 5\\
\begin{block}{(cc|ccc)}
1 & 0 & 1 & 0 & 1\\
0 & 1 & 0 & 1 & 1\\
\end{block}\end{blockarray},$$
and they both are characteristic maps over $\cP_5$.
\end{example}
We observe that from a pair $(K,\lambda)$ with $K$ a simplicial complex and $\lambda\in\CM(K)$, we get the pairs $(\Lk_K(v),\Proj_v\lambda)$, for $v$ any vertex of $K$. The operators $\Lk$ and $\Proj$ work together as a pair $(\Lk,\Proj)$.

The idea of Choi and Park puzzle method relies on what we computed on \Cref{ex_pentagon}. If we take a wedged simplicial $K(J)$ there are numerous copies of $K$ in $K(J)$ obtained after some link operations. We simultaneously consider a characteristic map $\Lambda\in\CM(K(J))$, then the consecutive projections of $\Lambda$ with respect to the same vertex sequence will provide a characteristic map $\lambda$ over $K$ as follows $(K(J),\Lambda)\stackrel{(\Lk,\Proj)}{\longrightarrow}\cdots\stackrel{(\Lk,\Proj)}{\longrightarrow}(K,\lambda)$.

Let us now review the definition of the (pre)diagram $D^\prime(K)$ of $K$.

\begin{definition}[(Pre)diagram of a PL sphere] Let $K$ be a PL sphere.
    We define the sets of the vertices ($\cV$), edges ($\cE$), and realizable squares ($\cS$) as follows.
\begin{enumerate}
\item $\cV =\CM(K)$
\item $\cE =$ the set of all characteristic maps in $\CM(\wed_p(K))$ for all vertices $p$ of $K$.
\item $\cS =$ the set of all characteristic maps in $\CM(\wed_{p,q}(K))$ for all distinct unordered pairs of vertices $(p,q)$  of $K$.
\end{enumerate}

The pair $(\cV,\cE)$ is called the \emph{prediagram} of $K$ and is denoted by $D^\prime(K)$.
The triplet $(\cV,\cE,\cS)$ is called the \emph{diagram} of $K$ and is denoted by $D(K)$.
\end{definition}

Now, here are some important facts concerning the (pre)diagram of a PL sphere.
\begin{enumerate}
\item Any edge $\lambda\in \cE$ can be associated to the triplet $(\lambda_1,\lambda_2,p)$, where $\lambda_1,\lambda_2\in \cV$ are two projections of $\lambda$, namely,
\begin{center}
$\lambda_1 = \Proj_{p_2}(\lambda)$,
$\lambda_2 = \Proj_{p_1}(\lambda)$,
 and we write \begin{tikzpicture}[baseline = (lambda.base),node distance=2cm, auto]
\node (lambda) {$\lambda_1$};
\node (lambda_prime) [right of = lambda] {$\lambda_2$};
\draw[-] (lambda) to node {$p$} (lambda_prime);
\end{tikzpicture}.
\end{center}
We thus say that $\lambda_1$ and $\lambda_2$ are \emph{$p$-adjacent}. An edge is said to be trivial if $\lambda_1=\lambda_2$. Every trivial edge $(\lambda,\lambda,p)$ is in $\cE$.
\item Just like for edges, any realizable square $\lambda\in \cS$ can be associated to the following projections onto $K$.
\begin{center}
\begin{tabular}{ccc}
\begin{tabular}{l}
$\lambda_{1,1} = \Proj_{p_2,q_2}(\lambda),$\\
$\lambda_{1,2} = \Proj_{p_2,q_1}(\lambda),$\\
$\lambda_{2,1} = \Proj_{p_1,q_2}(\lambda),$\\
$\lambda_{2,2} = \Proj_{p_1,q_1}(\lambda),$\\
\end{tabular} & and we write &
\begin{tikzpicture}[baseline =(anchor.base), node distance=2 cm, auto]
\node (lambda1) {$\lambda_{1,1}$};
\node (lambda2) [right of = lambda1] {$\lambda_{2,1}$};
\node (lambda3) [below of = lambda1] {$\lambda_{1,2}$};
\node (lambda4) [below of = lambda2] {$\lambda_{2,2}$};
\draw[-] (lambda1) to node {$p$} (lambda2);
\draw[-] (lambda1) to node (anchor) [left]  {$q$} (lambda3);
\draw[-] (lambda3) to node [below] {$p$} (lambda4);
\draw[-] (lambda2) to node {$q$} (lambda4);
\end{tikzpicture},
\end{tabular}
\end{center}
\noindent with $\lambda_{i,j}\in \cV$ for $i,j=1,2$, and each edge being in $\cE$.
\end{enumerate}

The diagram $D(K)$ stores the ``rules'' which are required for completing the puzzle. But now, we need a board in order to play this puzzle game. Let us remind that the main objective is to create $\CM(L(J))$ from $\CM(L)$. The board should intuitively represent the adjacent relations between copies of $L$ inside $L(J)$, as in \Cref{ex_pentagon}.

Let $J=(j_1,\ldots,j_m)$ be an $m$-tuple of positive integers.
We define the graph $G(J)$ of $J$ as the $1$-skeleton of the simplicial complex $\Delta^J := \Delta^{j_1-1}\times \cdots \times \Delta^{j_m-1}$. Each edge of $G(J)$ can be written uniquely as $$e=v_1\times v_2\times\ldots\times v_{p-1}\times e_{p} \times v_{p+1} \times \ldots\times v_m,$$ where $v_i$ is a vertex of $\Delta^{j_i-1}, i=1,\ldots,m,i\neq p$ and $e_p$ is an edge of $\Delta^{j_p-1}$. We endow $G(J)$ with an edge coloring where each $e$ is colored with $p\in [m]$.
\begin{remark} A subsquare of $G(J)$ is a subgraph of $G(J)$ that comes from a $2$-face of $\Delta^J$ which has 2 edges. We can see that the opposed edges of a square have pairwise the same color.
\end{remark}

\begin{example} Put $J = (3,2,1)$. The graph $G(J)$ is the $1$-skeleton of $\Delta^2\times\Delta^1\times\Delta^0$ we label its nodes as $(a_i,b_j,c_k)$ with $i=1,2,3$, $j=1,2$, and $k=1$.  One can notice that since we did not perform any wedge operation on the vertex $3$, there is no edges colored as $3$. We represent it as follows.
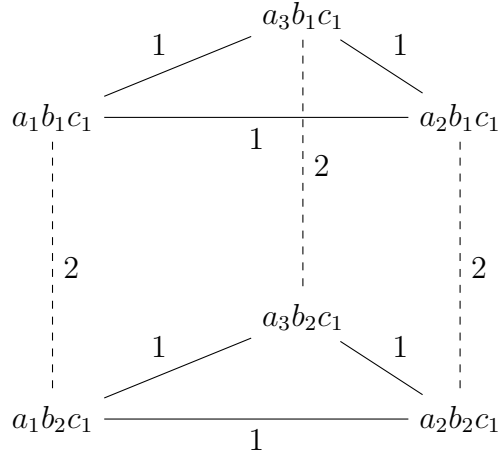
\begin{center}

\begin{tikzpicture}[node distance=4 cm, auto]
\node (a1b1c1) {$a_1 b_1 c_1$};
\node[right = of a1b1c1] (a2b1c1) {$a_2 b_1 c_1$};
\node[below of = a1b1c1] (a1b2c1)  {$a_1 b_2 c_1$};
\node[below of = a2b1c1] (a2b2c1)  {$a_2 b_2 c_1$};
\node[above left =1cm of a2b1c1] (a3b1c1)  {$a_3 b_1 c_1$};
\node[below of = a3b1c1] (a3b2c1)  {$a_3 b_2 c_1$};

\draw[-] (a1b1c1) to node [below] {$1$} (a2b1c1);
\draw[-] (a1b1c1) to node {$1$} (a3b1c1);
\draw[-] (a2b1c1) to node [above right] {$1$} (a3b1c1);

\draw[-] (a1b2c1) to node [below] {$1$} (a2b2c1);
\draw[-] (a1b2c1) to node {$1$} (a3b2c1);
\draw[-] (a2b2c1) to node [above right] {$1$} (a3b2c1);

\draw[-,dashed] (a1b1c1) to node {$2$} (a1b2c1);
\draw[-,dashed] (a2b1c1) to node {$2$} (a2b2c1);
\draw[-,dashed] (a3b1c1) to node {$2$} (a3b2c1);
\end{tikzpicture}
\end{center}

\end{example}

A \emph{realizable puzzle} is a coloring of the nodes of $G(J)$ by DJ classes of $K$ by respecting the edges and squares coloring rules given by the diagram of $K$. Namely, it is a map $\pi \colon \V(G(J)) \to \cV=\CM(K)$, which is in addition an edge coloring preserving pseudo-graph homomorphism $\pi \colon G(J) \to D^\prime(K)$ such that each image of a subsquare of $G(J)$ is a realizable square of $D(K)$.
The set of realizable puzzles will be denoted as $\cR\cP(K,J)$.

We thus have the following theorem.

\begin{theorem}[\cite{CP_wedge_2}, Theorem 5.4]
For any PL sphere $K$, there is a one-to-one correspondence
$$
    \CM(K(J)) \stackrel{1:1}{\longleftrightarrow} \cR\cP(K,J).
$$
\end{theorem}

\begin{remark} In this theorem, we consider the set of realizable puzzles of $K$ with wedge $J$. To compute it, we need to compute the set of characteristic maps over $K$ by the Garrison and Scott algorithm. We thus need $K$ to be of minimal dimension for the Garrison and Scott algorithm to finish quickly (see \Cref{complexity}), namely $K$ should be a seed.
\end{remark}

The puzzle method gives us directly the following algorithm for finding $\CM(L(J))$, for $L$ a seed.
Its main idea is to enumerate all possible puzzles, namely put every possible combination of stone positions on the board, and then check if some combinations make the puzzle realizable.
\begin{algorithm}[The old puzzle algorithm] \label{old_puzz}
 \begin{itemize}
    \item[ ]
    \item \textbf{Input:} $L$ a seed PL sphere on $[m]$ with $\dim(L)=n-1$, $J \in \Z_{>0}^m$ and $\CM(L)$ the set of the DJ classes over $L$.
    \item \textbf{Output:} The set $\cR \cP(L,J)$ of all realizable puzzles $\pi:G(J)\to D^\prime(L)$.
    \item \textbf{Initialization:}
        \begin{itemize}
            \item[] $D(L)\leftarrow$ diagram of $L$.
            \item[] $\pi\leftarrow$ The empty map.
            \item[] $\cR \cP (L,J)\leftarrow \emptyset$.
            \item[] $\mathcal{PT}(L,J)\leftarrow \emptyset$. (Puzzle Trash)
        \end{itemize}
    \item \textbf{Procedure:}
        \begin{enumerate}
            \item For each node $v$ of the graph $G(J)$, set $\pi(v) = \lambda$ for some $\lambda\in \CM(L)$. The puzzle $\pi$ should not be in $\mathcal{PT}(L,J)$.
            \item If the image of some edge of $G(J)$ by $\pi$ is not in $D(L)$ then add $\pi$ to $\mathcal{PT}(L,J)$ and go to (1).
            \item If the image of some subsquare of $G(J)$ by $\pi$ is not realizable then add $\pi$ to $\mathcal{PT}(L,J)$ and go to (1).
            \item Add $\pi$ to $\cR \cP(L,J)$.
		\end{enumerate}
\end{itemize}
\end{algorithm}

One can notice that the problem of this puzzle algorithm is that it needs to check if all the edges and the subsquares of the puzzle are in the diagram (Procedure (2) and(3)), and this for each possible puzzle.
That requires a lot of computation time (see \Cref{old_vs_new_puzz} for a detailed computation of the complexity of this algorithm), and we thus need a drastic improvement.

In the following section, we will create a constructive and procedural algorithm for the puzzle method to become applicable and faster.

\section{The improved puzzle algorithm}

\subsection{Computation of the edge rules}

Let us understand more explicitly what an edge of the prediagram represents.

First of all, let us introduce some notations.
For $\lambda\in\CM(K)$ and $k=1,\ldots,n$, we may assume that $\lambda(k) = e_k$, the $k$th vector of the canonical basis, and we will use the following notations.
\begin{align*}
\lambda &= \begin{blockarray}{ccccccccccc}
&1      & \cdots  & k                        & \cdots & n   & n+1 & \cdots & n+j & \cdots & m \\
\begin{block}{c(ccccc|ccccc)}
1&1      & 0       & \cdots  & \cdots & 0      &   \lambda(n+1)_1 & \cdots &  \lambda(n+j)_1 & \cdots &  \lambda(m)_1\\
&0 & \ddots  & \ddots  &  & \vdots &   \vdots       &   &  \vdots     &   &  \vdots \\
k&\vdots & \ddots  & 1       & \ddots & \vdots &   \lambda(n+1)_k & \cdots &  \lambda(n+j)_k & \cdots &  \lambda(m)_k\\
&\vdots &   & \ddots  & \ddots & 0 &   \vdots       &   &  \vdots     &   &  \vdots \\
n&0      & \cdots  & \cdots  & 0 & 1      &   \lambda(n+1)_n & \cdots &  \lambda(n+j)_n & \cdots &  \lambda(m)_n\\
\end{block}
\end{blockarray}\\
&=\begin{blockarray}{ccccccccccc}
&1      & \cdots  & k                        & \cdots & n   & n+1 & \cdots & n+j & \cdots & m \\
\begin{block}{c(ccccc|ccccc)}
1&1      & 0       & \cdots  & \cdots & 0      &   \bar{\lambda}(1)_{1} & \cdots &  \bar{\lambda}(1)_j & \cdots &  \bar{\lambda}(1)_{m-n}\\
&0 & \ddots  & \ddots  &  & \vdots &   \vdots       &   &  \vdots     &   &  \vdots \\
k&\vdots & \ddots  & 1       & \ddots & \vdots &   \bar{\lambda}(k)_{1} & \cdots &  \bar{\lambda}(k)_j & \cdots &  \bar{\lambda}(k)_{m-n}\\
&\vdots &   & \ddots  & \ddots & 0 &   \vdots       &   &  \vdots     &   &  \vdots \\
n&0      & \cdots  & \cdots  & 0 & 1      &   \bar{\lambda}(n)_{1} & \cdots &  \bar{\lambda}(n)_j & \cdots &  \bar{\lambda}(n)_{m-n}\\
\end{block}
\end{blockarray},
\end{align*}
for $j=1,\ldots,m-n$. Note that it coincides with the dual characteristic map notation.
\label{constr_puzzle}
Let $\lambda\in \CM(\wed_p(K))$, and let us denote its two projections by $\lambda_1 = \left(\begin{array}{c|c}& \bar{\lambda}_1(1) \\ I_n & \vdots \\ & \bar{\lambda}_1(n)\end{array}\right)$ and $\lambda_2 = \left(\begin{array}{c|c}& \bar{\lambda}_2(1) \\ I_n & \vdots \\ & \bar{\lambda}_2(n)\end{array}\right)$. They are $p$-adjacent by definition.

On one hand, in the case $p\in \{1,\ldots,n\}$, the characteristic map $\lambda$ which makes $\lambda_1$ and $\lambda_2$ $p$-adjacent can be written with the following form, called the basic form of $\lambda$ (defined in \cite{CP_wedge_2}).
$$\lambda=
\begin{blockarray}{cccccccccc}
& 1      & 2      &   & p_1    & p_2    &  & n-1    & n      & n+1, \ldots, m \\
\begin{block}{c(cccccccc|c)}
1 & 1      & 0      & \cdots  & 0      & 0      & \cdots & \cdots     & 0      & \bar{\lambda}_1(1) = \bar{\lambda}_2(1) \\
2 & 0      & 1      & \ddots  & \vdots & \vdots &  &  & \vdots &  \bar{\lambda}_1(2) = \bar{\lambda}_2(2) \\
&\vdots & \ddots & \ddots  & 0      & \vdots &  &  & \vdots & \vdots \\
p_1&\vdots &  & \ddots  & 1      & 0      &  &  & \vdots & \bar{\lambda}_1(p) \\
p_2&\vdots &  &   & 0      & 1      & \ddots &  & \vdots & \bar{\lambda}_2(p) \\
&\vdots &  &   & \vdots & 0      & \ddots & \ddots & \vdots & \vdots \\
n-1&\vdots &  &   & \vdots & \vdots & \ddots & 1      & 0      & \bar{\lambda}_1(n-1) = \bar{\lambda}_2(n-1) \\
n&0      & \cdots & \cdots  & 0 & 0 & \cdots & 0      & 1      & \bar{\lambda}_1(n) = \bar{\lambda}_2(n)\\
\end{block}
\end{blockarray},
$$
if we denote by $\phi := \bar{\lambda}_1(p) + \bar{\lambda}_2(p)$ the row vector corresponding to the change at row $p$ between $\bar{\lambda}_1$ and $\bar{\lambda}_2$. The edge in the pre-diagram can be seen as follows.

\begin{equation} \label{edge1}
\lambda_1 = \begin{array}{ccc}
\left(\begin{array}{c|c}& \bar{\lambda}_1(1) \\
& \vdots \\
I_n & \bar{\lambda}_1(p) \\
& \vdots \\
& \bar{\lambda}_1(n)\end{array}\right)
& \xrightarrow{\phi,p}
&
\lambda_2=\left(\begin{array}{c|c}& \bar{\lambda}_1(1) \\
& \vdots \\
I_n & \bar{\lambda}_1(p) + \phi \\
& \vdots \\
& \bar{\lambda}_1(n)\end{array}\right)
\end{array}.
\end{equation}

Conversely, one can see that $\lambda_2 \stackrel{\phi, p}{\longrightarrow} \lambda_1$ as well.

On the other hand, if $p>n$, we can find $1 \leq k \leq n$ such that the $k$th coordinate $\lambda_1(p)_k$ is equal to $1$. We can then reduce $\lambda_1(p)$ onto $e_k$ and thus $\lambda_1(k)$ becomes $\lambda_1(p)$.
If we write $a_j = \bar{\lambda}_1(k)_{j} = \lambda_1(n+j)_k$, $j=1,\ldots,m-n$, we have
$$
    \lambda_1 =
\begin{blockarray}{cccccccccc}
&1      & 2      & \cdots  & k                        & \cdots & n-1    & n      & p       & (n+j) \neq p \\
\begin{block}{c(ccccccc|c|c)}
1&1      & 0      & \cdots  & \lambda_1(p)_1     & \cdots & \cdots     & 0      & 0       & \bar{\lambda}_1(1)_j + \lambda_1(p)_1 a_j \\
2&0      & 1      & \ddots  & \vdots     &  &  & \vdots &  \vdots &  \bar{\lambda}_1(2)_j + \lambda_1(p)_2 a_j \\
&\vdots & \ddots & \ddots  & \vdots &  &  & \vdots &  0      & \vdots \\
k&\vdots &  & \ddots  & 1 = \lambda_1(p)_k &  &  & \vdots &  1      & a_j \\
&\vdots &  &   & \vdots & \ddots & \ddots & \vdots &  0      & \vdots \\
n-1&\vdots &  &   & \vdots                   & \ddots & 1      & 0      &  \vdots & \bar{\lambda}_1(n-1)_j + \lambda_1(p)_{n-1} a_j\\
n&0      & \cdots & \cdots  & \lambda_1(p)_n     & \cdots & 0      & 1      &  0      & \bar{\lambda}_1(n)_j + \lambda_1(p)_{n} a_j\\
\end{block}
\end{blockarray}.
$$

If we switch the columns $k$ and $p$, we get the same as previously after wedging
$$
\lambda=
\begin{blockarray}{ccccccccccc}
& 1      & 2      &   & p_1    & p_2    &  & n-1    & n     & k & (n+j) \neq p \\
\begin{block}{c(cccccccc|c|c)}
1 & 1      & 0      & \cdots  & 0      & 0      & \cdots & \cdots     & 0  & \lambda_1(p)_1    & \bar{\lambda}_1(1)_j + \lambda_1(p)_1 a_j \\
2 & 0      & 1      & \ddots  & \vdots & \vdots &   & & \vdots & \vdots & \bar{\lambda}_1(2)_j + \lambda_1(p)_2 a_j \\
&\vdots & \ddots & \ddots  & 0      & \vdots &  &  & \vdots & \vdots& \vdots \\
k_1&\vdots &  & \ddots  & 1      & 0      &  &  & \vdots & 1 & a_j \\
k_2&\vdots &  &   & 0      & 1      & \ddots &  & \vdots&b_{0} & b_j \\
&\vdots &  &   & \vdots & 0      & \ddots & \ddots & \vdots & \vdots& \vdots \\
n-1&\vdots &  &   & \vdots & \vdots & \ddots & 1      & 0 & \vdots     & \bar{\lambda}_1(n-1)_j + \lambda_1(p)_{n-1} a_j \\
n&0      & \cdots & \cdots  & 0 & 0 & \cdots & 0      & 1  &\lambda_1(p)_n    & \bar{\lambda}_1(n)_j + \lambda_1(p)_{n} a_j\\
\end{block}
\end{blockarray},
$$
with $b_{0}=1$ because otherwise $\{1,\ldots,k,\ldots,n\}$ is not a face of $\wed_p(K)$ anymore. We then perform the projection on the column $p_1$ and switch columns $p$ and $k$ again, we obtain
$$ \lambda_2 =
\begin{blockarray}{cccccccccc}
&1      & 2      & \cdots  & k                        & \cdots & n-1    & n      & p       & (n+j)\neq p \\
\begin{block}{c(ccccccc|c|c)}
1&1      & 0      & \cdots  & \lambda_1(p)_1     & \cdots & \cdots     & 0      & 0       & \bar{\lambda}_1(1)_j + \lambda_1(p)_1 a_j \\
2&0      & 1      & \ddots  & \vdots     &  &  & \vdots &  \vdots &  \bar{\lambda}_1(2)_j + \lambda_1(p)_2 a_j \\
&\vdots & \ddots & \ddots  & \vdots &  &  & \vdots &  0      & \vdots \\
k&\vdots &  & \ddots  & 1 = \lambda_1(p)_k &  &  & \vdots &  1      & b_j \\
&\vdots &  &   & \vdots & \ddots & \ddots & \vdots &  0      & \vdots \\
n-1&\vdots &  &   & \vdots                   & \ddots & 1      & 0      &  \vdots & \bar{\lambda}_1(n-1)_j + \lambda_1(p)_{n-1} a_j\\
n&0      & \cdots & \cdots  & \lambda_1(p)_n     & \cdots & 0      & 1      &  0      & \bar{\lambda}_1(n)_j + \lambda_1(p)_{n} a_j\\
\end{block}
\end{blockarray},$$
and so if we reduce again to have $\lambda_2(k)$ being equal to $e_k$ we get

$$ \lambda_2 =
\begin{blockarray}{cccccccccc}
&1      & 2      & \cdots  & k                        & \cdots & n-1    & n      & p       & (n+j) \neq p \\
\begin{block}{c(ccccccc|c|c)}
1&1      & 0      & \cdots  & 0    & \cdots & \cdots     & 0      & \lambda_1(p)_1       & \bar{\lambda}_1(1)_j + \lambda_1(p)_1 (a_j + b_j) \\
2&0      & 1      & \ddots  & \vdots     &  &  & \vdots &  \vdots &  \bar{\lambda}_1(2)_j + \lambda_1(p)_2 (a_j + b_j) \\
&\vdots & \ddots & \ddots  & \vdots &  &  & \vdots &  \vdots      & \vdots \\
k&\vdots &  & \ddots  & 1 &  &  & \vdots &  1      & b_j \\
&\vdots &  &   & \vdots & \ddots & \ddots & \vdots &  \vdots      & \vdots \\
n-1&\vdots &  &   & \vdots                   & \ddots & 1      &0      &  \vdots & \bar{\lambda}_1(n-1)_j + \lambda_1(p)_{n-1} (a_j + b_j)\\
n&0      & \cdots & \cdots  & 0     & \cdots & 0      & 1      &  \lambda_1(p)_n      & \bar{\lambda}_1(n)_j + \lambda_1(p)_{n} (a_j + b_j)\\
\end{block}
\end{blockarray}.$$

Thus if we write $\psi_{j} := a_j + b_j$, for $j=1,\ldots,m-n$ and $(n+j)\neq p$, and $\psi_{p}=0$ for $(n+j)=p$, the edge in the pre-diagram can be seen as follows

\begin{equation} \label{edge2}
\lambda_1=\begin{array}{cccc}
\left(\begin{array}{c|c}& \bar{\lambda}_1(1) \\
I_n & \vdots \\
& \bar{\lambda}_1(n)\end{array}\right)
& \xrightarrow{\psi,p}
&
\lambda_2=\left(\begin{array}{c|c}& \bar{\lambda}_1(1) + \lambda_1(p)_1\psi \\
I_n& \vdots \\
& \bar{\lambda}_1(n) + \lambda_1(p)_n\psi \end{array}\right)
\end{array},
\end{equation}
and $\lambda_2 \stackrel{\psi, p}{\longrightarrow} \lambda_1$.

Thus, if two DCM (dual characteristic maps) $\bar{\lambda}_1$ and $\bar{\lambda}_2$ are $p$-adjacent, the implying edge in the diagram can be associated to a unique (since in our calculations we chose a unique representative for each dual DJ class) row vector $\phi$ in \eqref{edge1} or $\psi$ in \eqref{edge2}.

Conversely, if we take $\lambda_1,\lambda_2\in \CM(K)$ such that there exists a relation of the form \eqref{edge1} or \eqref{edge2} according to a vertex $p$, then processing the same calculations in the opposite direction allows us to build a characteristic map $\lambda$ over $\wed_p(K)$ with $\{\lambda(p_1),\lambda(p_2)\}$ being a unimodular set. We now have the following proposition in \cite{Choi-Park2016}.
\begin{proposition}[Proposition~4.4 \cite{Choi-Park2016}]
Let $K$ be a PL sphere with vertex set $V$. For $v\in V$ and let $\lambda$ be a characteristic map on $\wed_v(K)$ such that $\{\lambda(v_1),\lambda(v_2)\}$ is a unimodular set. Then $\lambda$ is non-singular if and only if $\Proj_{v_1}(\lambda)$ and $\Proj_{v_2}(\lambda)$ are non-singular. Furthermore, $\lambda$ is uniquely determined by $\Proj_{v_1}(\lambda)$ and $\Proj_{v_2}(\lambda)$.
\end{proposition}
Since $\lambda_1$ and $\lambda_2$ are chosen non-singular, then $\lambda$ is non-singular. We write $\lambda:=\lambda_1 \wedge_{v} \lambda_2$ and we call it the \emph{$v$-wedge} of $\lambda_1$ and $\lambda_2$.

It then provides us the following lemma concerning the edges of the diagram.

\begin{lemma}\label{lemma_edge}
Let $K$ be a PL sphere on $[m]$ with $\dim(K)=n-1$, $\lambda_1,\lambda_2\in \CM(K)$ and $p\in[m]$. Then, the following are equivalent
\begin{enumerate}
\item There exists $\lambda \in \CM(\wed_p(K))$ such that $\lambda = \lambda_1\wedge_{p}\lambda_2$;
\item There exists a $\Z_2$ row vector $\psi$ of size $(m-n)$ such that $\bar{\lambda}_2 = h\bar{\lambda}_1$\\
with $h= \left(\begin{array}{c|c}
I_n & \begin{array}{c} \lambda_1(p)_1\psi \\ \vdots \\ \lambda_1(p)_i\psi \\ \vdots \\ \lambda_1(p)_n\psi \end{array} \\ \hline
0   & I_{m-n}\end{array} \right) \in \cM(m,m,\Z_2)$, and $\bar{\lambda}_i = \left(\begin{array}{c} \bar{\lambda}_i(1) \\ \vdots \\ \bar{\lambda}_i(n) \\ \hline I_{m-n}\end{array}\right)$, for $i=1,2$.
\end{enumerate}
\end{lemma}

\begin{proof}The proof comes directly when one notices that the transformations \eqref{edge1} (since $\lambda_1(p) = e_p, p=1,\ldots,n$) and \eqref{edge2} can be seen as follows.
$$
\left(\begin{array}{c} \bar{\lambda}_1(1) + \lambda_1(p)_1\psi \\
\vdots \\
\bar{\lambda}_1(i) + \lambda_1(p)_i\psi \\
\vdots \\
\bar{\lambda}_1(n) + \lambda_1(p)_n\psi \\
\hline I_{m-n}
\end{array}\right) = \left(\begin{array}{c|c}
I_n & \begin{array}{c}\lambda_1(p)_1\psi \\ \vdots \\ \lambda_1(p)_i\psi \\ \vdots \\ \lambda_1(p)_n\psi \end{array} \\ \hline
0   & I_{m-n}\end{array} \right) \left(\begin{array}{c} \bar{\lambda}_1(1) \\
\vdots \\
\bar{\lambda}_1(i) \\
\vdots \\
\bar{\lambda}_1(n)\\
\hline I_{m-n}\end{array}\right) \Leftrightarrow \bar{\lambda}_2 = h \bar{\lambda}_1,$$
with $h$ of the desired form.
\end{proof}

\begin{remark}This construction can be clarified as follows. Two dual DJ classes $\bar{\lambda}_1$ and $\bar{\lambda}_2$ are $p$-adjacent if and only if we can go from $\bar{\lambda}_1$ to $\bar{\lambda}_2$ by adding the same row vector $\psi$ to each row of $\bar{\lambda}_1$ at coordinate $j$ which satisfies $\lambda_1(p)_j=1$. One can then see that from any pair of characteristic maps over $K$, it is very easy to check if they are $p$-adjacent or not, for some vertex $p$ of $K$, using linear algebra. Furthermore, this construction gives us a very direct way of computing the prediagram of $K$ simply from $\CM(K)$.
\end{remark}

For all that follows, since the matrices $h$ depend on $p$, $\lambda_1$ and $\psi$, we will denote such $h$ by a triplet $(p,\lambda_1,\psi)$.

For a dual characteristic map $\bar{\lambda}_0 \in \DCM(K)$, let us write $H_p(\bar{\lambda}_0)$ the set of the matrices $h = (p,\lambda_0,\psi)$ from all the edges \begin{tikzpicture}[baseline = (lambda.base),node distance=1.5cm, auto]
\node (lambda) {$\bar{\lambda}_0$};
\node (lambda_prime) [right of = lambda] {$h\bar{\lambda}_0$};
\draw[-] (lambda) to node {$p$} (lambda_prime);
\end{tikzpicture} which are in the prediagram.
From now on, we will write such an edge \begin{tikzpicture}[baseline = (lambda.base),node distance=2cm, auto]
\node (lambda) {$\bar{\lambda}_0$};
\node (lambda_prime) [right of = lambda] {$h\bar{\lambda}_0$};
\draw[-] (lambda) to node {$p,h$} (lambda_prime);
\end{tikzpicture}.

\begin{remark} \label{remark_unchanged_colors} Here are some important properties of $H_p(\bar{\lambda}_0)$.
\begin{enumerate}
\item It is easy to check that this set is actually a commutative subgroup of $GL(m,\Z_2)$ whose elements act from the left on the set of the $\bar{\lambda}\in \DCM(K)$ which are $p$-adjacent to $\bar{\lambda}_0$.
\item  Furthermore, one can notice that if $q\neq p$ is another vertex of $K$, and if $\lambda_0(p) = \lambda_0(q)$, then if $h = (p,\lambda_0,\psi) \in H_p(\bar{\lambda}_0)$ we have $h^\prime = (q,\lambda_0,\psi) \in H_q(\bar{\lambda}_0)$ so $\psi(p)=\psi(q)=0$ which means that elements of $H_p(\bar{\lambda}_0)$ cannot modify the color of the vertices having the same color as $\lambda_0(p)$ and since the roles of $p$ and $q$ can be switched, $H_q(\bar{\lambda}_0) = H_p(\bar{\lambda}_0)$.
\item The action of $H_p(\bar{\lambda}_0)$ on the $p$-adjacent component of $\bar{\lambda}_0$ is free (and transitive) by construction. Therefore if $\bar{\lambda}_1$ and $\bar{\lambda}_2$ are $p$-adjacent, then $H_p(\bar{\lambda}_1) = H_p(\bar{\lambda}_2)$.
\end{enumerate}
\end{remark}

Choi and Park found that a realizable puzzle is automatically entirely determined if one chooses the image of a node $v$ of $G(J)$ and of every nodes $w$ connected to $v$ by an edge as in the following proposition.
\begin{proposition}[\cite{CP_wedge_2}, Proposition 4.3]\label{CP_nice_prop} Let $v$ be any fixed node of $G(J)$ and suppose that there are two realizable puzzles $\pi,\pi^\prime:G(J)\to D^\prime(K)$. Then $\pi = \pi^\prime$ if and only if $\pi(v)=\pi^\prime(v)$ and $\pi(w) = \pi^\prime(w)$ for every node $w$ of $G(J)$ which is adjacent to $v$.
\end{proposition}

The latter proposition is the cornerstone of the new algorithm for finding realizable puzzles in a constructive way. In fact, when we have a seed PL sphere $L$ on $[m]$ and an $m$-tuple $J$, we just need the image by a realizable puzzle $\pi$ of few nodes of $G(J)$ to get the entire DJ class corresponding to $\pi$.

\subsection{Computation of the square rules}

Let us recall that, from Theorem~5.4 in \cite{CP_wedge_2}, we know a puzzle is realizable if and only if all of its subsquares are realizable. So let us find a condition for the squares to be realizable. More precisely, since we want to describe a constructive algorithm, we consider the smallest step we can make in the construction of the realizable puzzle.
This step is completing a square which would have one of its nodes ``missing''.

An \emph{incomplete} $pq$-square is a square of the form
\begin{center}
\begin{tikzpicture}[baseline =(anchor.base), node distance=3 cm, auto]
\node (lambda1) {$\bar{\lambda}$};
\node (lambda2) [right of = lambda1] {$\bar{\lambda}^{\alpha}$};
\node (lambda3) [below of = lambda1] {$\bar{\lambda}^{\beta}$};
\node (lambda4) [below of = lambda2] {?};
\draw[-] (lambda1) to node {$p,\alpha$} (lambda2);
\draw[-] (lambda1) to node (anchor) [left]  {$q,\beta$} (lambda3);
\draw[-] (lambda3) to node [below] {$p,?$} (lambda4);
\draw[-] (lambda2) to node {$q,?$} (lambda4);
\end{tikzpicture}, with $p\neq q$.
\end{center}

Completing this square requires to
\begin{itemize}
\item Find an element $\bar{\lambda}^\gamma$ such that \begin{tikzpicture}[baseline = (lambda.base),node distance=2cm, auto]
\node (lambda) {$\bar{\lambda}^{\alpha}$};
\node (lambda_prime) [right of = lambda] {$\bar{\lambda}^{\gamma}$};
\draw[-] (lambda) to node {$q$} (lambda_prime);
\end{tikzpicture} and \begin{tikzpicture}[baseline = (lambda.base),node distance=2cm, auto]
\node (lambda) {$\bar{\lambda}^{\beta}$};
\node (lambda_prime) [right of = lambda] {$\bar{\lambda}^{\gamma}$};
\draw[-] (lambda) to node {$p$} (lambda_prime);
\end{tikzpicture} are in the prediagram of $K$. This can be easily done with a computer algorithm by \Cref{lemma_edge}. Such element will be called \emph{a possible missing piece} of the square.
\item Check if a possible missing piece $\bar{\lambda}^{\gamma}$ makes the square \begin{tikzpicture}[baseline =(anchor.base), node distance=1.5 cm, auto]
\node (lambda1) {$\bar{\lambda}$};
\node (lambda2) [right of = lambda1] {$\bar{\lambda}^{\alpha}$};
\node (lambda3) [below of = lambda1] {$\bar{\lambda}^{\beta}$};
\node (lambda4) [below of = lambda2] {$\bar{\lambda}^{\gamma}$};
\draw[-] (lambda1) to node {$p$} (lambda2);
\draw[-] (lambda1) to node (anchor) [left]  {$q$} (lambda3);
\draw[-] (lambda3) to node [below] {$p$} (lambda4);
\draw[-] (lambda2) to node {$q$} (lambda4);
\end{tikzpicture} a realizable square. We will call such $\bar{\lambda}^{\gamma}$  \emph{\underline{the} missing piece} of the square.
\end{itemize}

The missing piece of a square is unique, due to \Cref{CP_nice_prop} applied on the $J$ corresponding to $\wed_{pq}$, which makes the terminology correct.

We will prove the following lemma.

\begin{lemma}\label{lemma_square}
For an incomplete $pq$-square as before we have these three possible cases
\begin{enumerate}
\item If there exists no possible missing piece, then the square cannot be completed;
\item If there exists more than one possible missing pieces for the square, then we can explicitly compute and find the missing piece of the puzzle among them;
\item If there exists only a single possible missing piece for the square, then it is the missing piece of this square and we can once more compute it explicitly.
\end{enumerate}
\end{lemma}

\begin{proof} (1) is immediate by the definitions we used.

For proving (2), we consider that there exists $\bar{\lambda}^\gamma$ and $\bar{\lambda}^{\gamma\prime}$ two possible missing pieces of this square. We have
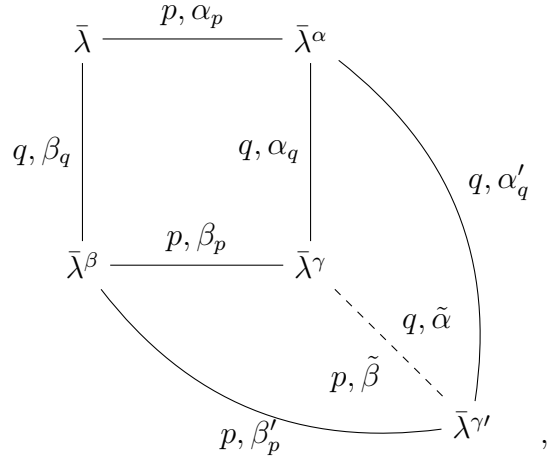
\begin{center}

\begin{tikzpicture}[baseline =(anchor.base), node distance=3 cm, auto]
\node (lambda1) {$\bar{\lambda}$};
\node (lambda2) [right of = lambda1] {$\bar{\lambda}^{\alpha}$};
\node (lambda3) [below of = lambda1] {$\bar{\lambda}^{\beta}$};
\node (lambda4) [below of = lambda2] {$\bar{\lambda}^\gamma$};
\node (lambda5) [below right of  = lambda4] {$\bar{\lambda}^{\gamma\prime}$};
\draw[-] (lambda1) to node {$p,\alpha_p$} (lambda2);
\draw[-] (lambda1) to node (anchor) [left]  {$q,\beta_q$} (lambda3);
\draw[-] (lambda3) to node {$p, \beta_p$} (lambda4);
\draw[-] (lambda2) to node [left] {$q, \alpha_q$} (lambda4);
\draw[-,bend left] (lambda2) to node {$q, \alpha_q^\prime$} (lambda5);
\draw[-,bend right] (lambda3) to node (anchor) [below]  {$p, \beta_p^\prime$} (lambda5);
\draw[-,dashed] (lambda4) to node [below left] {$p, \tilde{\beta}$} node [above right] {$q, \tilde{\alpha}$} (lambda5);
\end{tikzpicture},
\end{center}
with $\alpha_p\in H_p(\bar{\lambda}),\beta_q\in H_q(\bar{\lambda})$ and $\alpha_q,\alpha_q^\prime\in H_q(\bar{\lambda}^\alpha),\beta_p,\beta_p^\prime\in H_p(\bar{\lambda}^\beta)$ and
 the dashed line being found as follows
\begin{itemize}
\item $\bar{\lambda}^{\gamma\prime}= \alpha_q^\prime \bar{\lambda}^\alpha = \alpha_q^\prime(\alpha_q \bar{\lambda}^\gamma) = \tilde{\alpha} \bar{\lambda}^\gamma$, so the edge \begin{tikzpicture}[baseline = (lambda.base),node distance=1.5cm, auto]
\node (lambda) {$\bar{\lambda}^{\gamma}$};
\node (lambda_prime) [right of = lambda] {$\bar{\lambda}^{\gamma\prime}$};
\draw[-] (lambda) to node {$q, \tilde{\alpha}$} (lambda_prime);
\end{tikzpicture}
 is in the prediagram of $K$ since $\tilde{\alpha}\in H_q(\lambda^{\gamma})$;
\item Similarly, \begin{tikzpicture}[baseline = (lambda.base),node distance=1.5cm, auto]\node (lambda) {$\bar{\lambda}^{\gamma}$};
\node (lambda_prime) [right of = lambda] {$\bar{\lambda}^{\gamma\prime}$};
\draw[-] (lambda) to node {$p, \tilde{\beta}$} (lambda_prime);
\end{tikzpicture} is in the prediagram of $K$ since $\tilde{\beta} = \beta_p^\prime \beta_p \in H_p(\lambda^{\gamma})$.
\end{itemize}

We write $\tilde{\alpha} = (q,\lambda^\gamma,\phi)$ and $\tilde{\beta} = (p,\lambda^\gamma,\psi)$.

Since by \Cref{constr_puzzle} a generator between two dual DJ classes of $K$ is unique, we have $\tilde{\alpha} = \tilde{\beta}$ and the equality between such objects is given by $\phi = \psi$ and the color of $p$, respectively $q$, in the $p$-adjacent, respectively $q$-adjacent component of $\bar{\lambda}^\gamma$ remains unchanged (by \Cref{remark_unchanged_colors} (2)). Thus since $\bar{\lambda}$ is in both of these connected components, this leads to $\lambda(p) = \lambda(q)$.

Hence the only case there would be several possible missing pieces is when $\lambda(p) = \lambda(q)$.
In that case, we can easily find which dual DJ class is the missing piece (i.e. creates a realizable square). It is exactly given by $\bar{\lambda}^\gamma := \alpha_p \beta_q \bar{\lambda}$. The corresponding square is
\begin{center}
\begin{tikzpicture}[baseline =(anchor.base), node distance=3 cm, auto]
\node (lambda1) {$\bar{\lambda}$};
\node (lambda2) [right of = lambda1] {$\bar{\lambda}^\alpha$};
\node (lambda3) [below of = lambda1] {$\bar{\lambda}^\beta$};
\node (lambda4) [below of = lambda2] {$\bar{\lambda}^\gamma$};
\draw[-] (lambda1) to node {$p,\alpha_p$} (lambda2);
\draw[-] (lambda1) to node (anchor) [left]  {$q,\beta_q$} (lambda3);
\draw[-] (lambda3) to node [below] {$p, \alpha_p$} (lambda4);
\draw[-] (lambda2) to node {$q,\beta_q$} (lambda4);
\end{tikzpicture}.
\end{center}
Let us show that it is a realizable square. We write $\alpha_p = (p,\lambda^\gamma,\phi)$ and $\beta_q = (q,\lambda^\gamma,\psi)$.
Without loss of generality, we can suppose that $p=1$ and $q=n+1$. We have

$$\bar{\lambda} = \left(\begin{array}{c}\bar{\lambda}(1) \\ \bar{\lambda}(\star) \\ I_{m-n}\end{array}\right),
\bar{\lambda}^\alpha = \left(\begin{array}{c}\bar{\lambda}(1) +\phi\\ \bar{\lambda}(\star) \\ I_{m-n}\end{array}\right),
\bar{\lambda}^\beta = \left(\begin{array}{c}\bar{\lambda}(1)+\psi \\ \bar{\lambda}(\star) \\ I_{m-n}\end{array}\right), \text{ and }
\bar{\lambda}^\gamma = \left(\begin{array}{c}\bar{\lambda}(1)+\psi+\phi \\ \bar{\lambda}(\star) \\ I_{m-n}\end{array}\right).$$
Thus
$$\lambda \wedge_{1} \lambda^\alpha = \begin{blockarray}{ccccc}
1_1 & 1_2 & 2,\ldots,n & q & q+1,\ldots,m\\
\begin{block}{(cc|c|c|c)}
1&0&0&1&\bar{\lambda}(1) + \phi \\
0&1&0&1&\bar{\lambda}(1) \\
\mathbf{0}&\mathbf{0}&I_{n-2} &\mathbf{0}&\bar{\lambda}(\star)\\
\end{block}\\
\end{blockarray},
$$
and
$$\lambda^\beta \wedge_{1} \lambda^\gamma = \begin{blockarray}{ccccc}
1_1 & 1_2 & 2,\ldots,n & q & q+1,\ldots,m \\
\begin{block}{(cc|c|c|c)}
1&0&0&1&\bar{\lambda}(1)+\psi + \phi \\
0&1&0&1&\bar{\lambda}(1)+\psi \\
\mathbf{0}&\mathbf{0}&I_{n-2} &\mathbf{0}& \bar{\lambda}(\star)\\
\end{block}
\end{blockarray}.
$$

By \Cref{lemma_edge}, the edge \begin{tikzpicture}[baseline = (lambda.base),node distance=2.5cm, auto]
\node (lambda) {$\lambda \wedge_{1} \lambda^\alpha$};
\node (lambda_prime) [right of = lambda] {$\lambda^\beta \wedge_{1} \lambda^\gamma$};
\draw[-] (lambda) to node {$q,\Delta$} (lambda_prime);
\end{tikzpicture} with $\Delta = (q,\lambda^\beta \wedge_{1} \lambda^\gamma,\psi)$ is in the prediagram of $\wed_p(K)$ and thus represents a DJ class of $\wed_q(\wed_p(K)) = \wed_{p,q}(K)$ and thus the square is realizable which proves (2).

Now, we need to prove (3). So let us consider that there exists only a unique possible missing piece. We have $\lambda(p)\neq\lambda(q)$ otherwise we can just look at the above calculations and easily compute the missing piece.

We then have two cases.

Firstly, if $p$ and $q$ are together in a maximal face of $K$. By relabeling and without loss of generality we can consider that $p=1$ and $q=2$. We consider the following $(1,2)$-square

\begin{center}
\begin{tikzpicture}[baseline =(anchor.base), node distance=3 cm, auto]
\node (lambda1) {$\bar{\lambda}$};
\node (lambda2) [right of = lambda1] {$\bar{\lambda}^\alpha$};
\node (lambda3) [below of = lambda1] {$\bar{\lambda}^\beta$};
\node (lambda4) [below of = lambda2] {$\bar{\lambda}^\gamma$};
\draw[-] (lambda1) to node {$1,\alpha_1$} (lambda2);
\draw[-] (lambda1) to node (anchor) [left]  {$2,\beta_2$} (lambda3);
\draw[-] (lambda3) to node [below] {$1, \alpha_1$} (lambda4);
\draw[-] (lambda2) to node {$2,\beta_2$} (lambda4);
\end{tikzpicture}.
\end{center}
Let us show that it is realizable. We write $\alpha_1 = (1,\bar{\lambda},\phi)$ and $\beta_2 = (2,\bar{\lambda},\psi)$, we can then describe the dual DJ classes
$$\bar{\lambda} = \left(\begin{array}{c}\bar{\lambda}(1) \\ \bar{\lambda}(2)\\ \bar{\lambda}(\star) \\ I_{m-n}\end{array}\right),
\bar{\lambda}^\alpha = \left(\begin{array}{c}\bar{\lambda}(1) +\phi\\ \bar{\lambda}(2) \\\bar{\lambda}(\star) \\ I_{m-n}\end{array}\right),
\bar{\lambda}^\beta = \left(\begin{array}{c}\bar{\lambda}(1) \\ \bar{\lambda}(2) +\psi \\\bar{\lambda}(\star) \\ I_{m-n}\end{array}\right), \text{ and }
\bar{\lambda}^\gamma = \left(\begin{array}{c}\bar{\lambda}(1)+\phi \\ \bar{\lambda}(2)+\psi\\ \bar{\lambda}(\star) \\ I_{m-n}\end{array}\right).$$

The $1$-wedges of this square are as follows

$$\lambda \wedge_{1} \lambda^\alpha = \begin{blockarray}{ccccc}
1_1 & 1_2 & 2& 3,\ldots,n&n+1,\ldots,m \\
\begin{block}{(cc|c|c|c)}
1&0&0&\mathbf{0}&\bar{\lambda}(1) + \phi \\
0&1&0&\mathbf{0}&\bar{\lambda}(1) \\
0&0&1&\mathbf{0}&\bar{\lambda}(2) \\
\mathbf{0}&\mathbf{0}&\mathbf{0}&I_{n-3}&\bar{\lambda}(\star)\\
\end{block}
\end{blockarray},$$
and
$$\lambda^\beta \wedge_{1} \lambda^\gamma = \begin{blockarray}{ccccc}
1_1 & 1_2 & 2& 3,\ldots,n&n+1,\ldots,m \\
\begin{block}{(cc|c|c|c)}
1&0&0&\mathbf{0}&\bar{\lambda}(1) + \phi \\
0&1&0&\mathbf{0}&\bar{\lambda}(1) \\
0&0&1&\mathbf{0}&\bar{\lambda}(2) + \psi \\
\mathbf{0}&\mathbf{0}&\mathbf{0}&I_{n-3}&\bar{\lambda}(\star)\\
\end{block}
\end{blockarray},$$

Again, \Cref{lemma_edge} gives that the edge \begin{tikzpicture}[baseline = (lambda.base),node distance=3cm, auto]
\node (lambda) {$\lambda \wedge_{1} \lambda^\alpha$};
\node (lambda_prime) [right of = lambda] {$\lambda^\beta \wedge_{1} \lambda^\gamma$};
\draw[-] (lambda) to node {$2,\Delta$} (lambda_prime);
\end{tikzpicture}, with $\Delta=(2,\lambda \wedge_{1} \lambda^\alpha,\psi)$, is in the prediagram of $\wed_1(K)$ and thus represents a DJ class of $\wed_2(\wed_1(K)) = \wed_{1,2}(K)$, namely, the square is realizable.

Secondly, if $p$ and $q$ appear in none of the maximal faces. By relabeling, we can suppose that $p=1$ and $q=n+1$.
Hence $\lambda(p) = e_1$ and $\lambda(q)$ is an arbitrary vector different from $e_1$ (We again forget this case since we have computed it before). We consider the following square.
\begin{center}
\begin{tikzpicture}[baseline =(anchor.base), node distance=3 cm, auto]
\node (lambda1) {$\bar{\lambda}$};
\node (lambda2) [right of = lambda1] {$\bar{\lambda}^\alpha$};
\node (lambda3) [below of = lambda1] {$\bar{\lambda}^\beta$};
\node (lambda4) [below of = lambda2] {$\bar{\lambda}^\gamma$};
\draw[-] (lambda1) to node {$1,(1,\lambda,\phi)$} (lambda2);
\draw[-] (lambda1) to node (anchor) [left]  {$n+1,(n+1,\lambda,\psi)$} (lambda3);
\draw[-] (lambda3) to node [below] {$1,(1,\lambda^\beta,\epsilon)$} (lambda4);
\draw[-] (lambda2) to node {$n+1,(n+1,\lambda^\alpha,\delta)$} (lambda4);
\end{tikzpicture}.
\end{center}

Recall that $\psi_{1}=\delta_{1}=0$ (see \eqref{edge2}), so inside this square, the only coordinate of $\lambda(n+1)$ which can change is its first coordinate.

Since $\lambda(n+1)\not\in \{\mathbf{0},e_1\}$, we must have an index $k>1$ such that $\lambda(n+1)_k=1$ (of course $n>1$), for simplifying the notations, we suppose that $k=2$, then $(n+1,\lambda,\psi)$ will modify at least $\bar{\lambda}(k)$.We first get
$$\bar{\lambda}^\alpha = \left(\begin{array}{c}\bar{\lambda}(1) +\phi \\\bar{\lambda}(k) \\ \bar{\lambda}(\star)\\ I_{m-n}\end{array}\right)\text{ and } \bar{\lambda}^\beta = \left(\begin{array}{c}\bar{\lambda}(1) + \lambda(n+1)_1\times\psi \\\bar{\lambda}(k) + \psi \\ \bar{\lambda}(\star) + \lambda(n+1)_\star\times\psi\\
I_{m-n}\end{array}\right).$$
From the two remaining edges, we obtain two forms for $\bar{\lambda}^\gamma$ as follows.
$$\bar{\lambda}^\gamma = \left(\begin{array}{c}\bar{\lambda}(1) + \phi + [\lambda(n+1)_1 + \phi_{1}]\times\delta \\\bar{\lambda}(k) + \delta \\ \bar{\lambda}(\star) + \lambda(n+1)_\star\times\delta\\ I_{m-n}\end{array}\right) \text{ and } \bar{\lambda}^\gamma = \left(\begin{array}{c}\bar{\lambda}(1) + \lambda(n+1)_1\times\psi+ \epsilon \\\bar{\lambda}(k) + \psi \\ \bar{\lambda}(\star) + \lambda(n+1)_\star\times\psi\\ I_{m-n}\end{array}\right).$$
Hence we get from the $k$th row that $\delta=\psi$. The first row provides the following explicit formula $\epsilon =  \lambda(n+1)_1\times\psi + \phi + [\lambda(n+1)_1 + \phi_{1}]\times\psi = \phi + \phi_1 \times\psi$, but will not be used here.
The first $1$-wedge of the square is
$$\lambda \wedge_{1} \lambda^\alpha = \begin{blockarray}{ccccc}
1_1 & 1_2 & 2,\ldots,n & n+1 & n+j=n+2,\ldots,m \\ \begin{block}{(cc|c|c|c)}
1&0&0&\lambda(n+1)_1 + \phi_{1}&\bar{\lambda}(1)_j + \phi_j \\
0&1&0&\lambda(n+1)_1&\bar{\lambda}(1)_j \\
\mathbf{0}&\mathbf{0}&I_{n-2} &\lambda(n+1)_\star&\bar{\lambda}(\star)_j\\
\end{block}
\end{blockarray}.
$$
The first coordinate of $\lambda^\gamma(n+1)$ is
\begin{align*}
\lambda^\gamma(n+1)_1 = \bar{\lambda}^\gamma(1)_1 &= [\bar{\lambda}(1)+\phi + (\lambda(n+1)_1+\phi_{1})\times\psi]_{1}\\ &= \bar{\lambda}(1)_1+\phi_{1} + (\lambda(n+1)_1+\phi_{1})\times\underbrace{\psi_{1}}_{=0}\\ &= \lambda(n+1)_1+\phi_{1}.
\end{align*}
Furthermore the other coordinates $\lambda^\gamma(n+1)_i$, for $i=2,\ldots,n$, are
\begin{align*}
\lambda^\gamma(n+1)_i = \bar{\lambda}^\gamma(i)_1 &= [\bar{\lambda}(i) + (\lambda(n+1)_1+\phi_{1})\times\psi]_{1}\\ &= \bar{\lambda}(i)_1 + (\lambda(n+1)_1+\phi_{1})\times\underbrace{\psi_{1}}_{=0}\\ &= \lambda(n+1)_i.
\end{align*}
From this we obtain the second $1$-wedge as follows
$$\lambda^\beta \wedge_{1} \lambda^\gamma = \begin{blockarray}{ccccc}
1_1 & 1_2 & 2,\ldots,n & n+1 & n+j=n+2,\ldots,m \\ \begin{block}{(cc|c|c|c)}
1&0&0&\lambda(n+1)_1 + \phi_{1} &\bar{\lambda}(1)_j+\phi_j + (\lambda(n+1)_1+\phi_{1})\times\psi_j \\
0&1&0&\lambda(n+1)_1&\bar{\lambda}(1)_j + \lambda(n+1)_1\times\psi_j \\
\mathbf{0}&\mathbf{0}&I_{n-2} &\lambda(n+1)_\star& \bar{\lambda}(\star)_j + \lambda(n+1)_\star\times\psi_j\\
\end{block}
\end{blockarray}.
$$
One can see by \Cref{lemma_edge} that the edge \begin{tikzpicture}[baseline = (lambda.base),node distance=3cm, auto]
\node (lambda) {$\lambda \wedge_{1} \lambda^\alpha$};
\node (lambda_prime) [right of = lambda] {$\lambda^\beta \wedge_{1} \lambda^\gamma$};
\draw[-] (lambda) to node {$n+1,\Delta$} (lambda_prime);
\end{tikzpicture}, with $\Delta=(n+1,\lambda \wedge_{1} \lambda^\alpha,\psi)$, is in the prediagram of $\wed_1(K)$ and thus provides a DJ class over $\wed_{1,n+1}(K)$ as desired.

Thus, in every cases we found that if there exists a possible missing piece for an incomplete $pq$-square, then it leads to the unique missing piece of the square.
\end{proof}

\begin{remark}This proof actually gives us the explicit and only form that realizable squares can take. Thus, while \Cref{lemma_edge} provides the set of edges $\cE$ of the diagram of $K$, \Cref{lemma_square} gives us the set of realizable squares $\cS$ of $D(K)$, namely all the DJ classes over $K$ after two consecutive wedges. More precisely, it gives us a map from the set of incomplete $pq$-squares onto $\CM(\wed_{p,q}(K))\cup\{\emptyset\}$ which gives either the missing piece of an incomplete $pq$-square if it can be completed or $\emptyset$ otherwise.
\end{remark}

As desired, \Cref{lemma_edge} and \Cref{lemma_square} provides us all the  data of the diagram of $K$. So we have now the basic steps of an algorithm which will enumerate all the realizable puzzles over a wedged seed PL sphere $L(J)$.

\subsection{Description of the new puzzle algorithm}

Let us recall that a node $v$ in the graph $G(J)$ can be represented by an $m$-tuple
$$(v_1,v_2,\ldots,v_m),$$
with $1\leq v_1\leq j_1,1\leq v_2\leq j_2,\ldots,1\leq v_m\leq j_m$.
We define the depth of a node $v$ in the graph $D(J)$ as $d(v) = |\{v_i\neq 1\vert i=1,\ldots,m\}|\leq m$.
For example we have $d((1,\ldots,1)) = 0$.
In addition, we order the nodes of $G(J)$ having the same depth by lexicographic order.
The depth $d$ together with the lexicographic order gives a total order $\preceq$ on the nodes of $G(J)$.

We will define by $D^\prime(L)_p$ the prediagram of $L$ which is restricted to the edges colored with $p$.
We will also denote by $D^\prime(L)_p(\lambda)$ the connected component of $\lambda$ in $D^\prime(L)_p$.
Since the action is free and transitive on a connected component, then $D^\prime(L)_p(\lambda)$ is a complete graph.

We then have the following theorem.

\begin{theorem}
Given a wedged seed PL sphere $L(J)$ and its set of DJ classes $\CM(L)$, there exists an algorithm for finding the DJ classes set $\CM(L(J))$ whose complexity only depends on the size of $\CM(L)$, the length of $J$, and the Picard number of $L$.
The algorithm is described in \Cref{new_puzz}.
\end{theorem}

\begin{algorithm}\label{new_puzz}
 \begin{itemize}
    \item[ ]
    \item \textbf{Input:} $L$ a seed PL sphere on $[m]$ with $\dim(L)=n-1$, $J \in \Z_{>0}^m$ and $\CM(L)$ the set of the DJ classes over $L$.
    \item \textbf{Output:} The set $\cR\cP(L,J)$ of all realizable puzzles $\pi:G(J)\to D(L)$ .
    \item \textbf{Initialization:}
        \begin{itemize}
            \item[] $D^\prime(L)\leftarrow$ Prediagram of $L$ using \Cref{lemma_edge}.
            \item[] $\pi\leftarrow$ The empty map.
            \item[] $\cR\cP(L,J)\leftarrow \emptyset$.
            \item[] $d \leftarrow 0$.
            \item[] $v \leftarrow (1,\ldots,1)$.
        \end{itemize}
    \item \textbf{Procedure:}
        \begin{enumerate}
            \item If $d>d(J)$ add $\pi$ to $\cR\cP(L,J)$;
            \item Else
            \begin{enumerate}
            \item  $\Lambda \leftarrow \CM(L)$
            \item  For every node $\nu\prec v$, if there is an edge \begin{tikzpicture}[baseline = (lambda.base),node distance=1.5cm, auto]
\node (lambda) {$\nu$};
\node (lambda_prime) [right of = lambda] {$v$};
\draw[-] (lambda) to node {$p$} (lambda_prime);
\end{tikzpicture} in $G(J)$ then
\begin{enumerate}
\item $\Lambda \leftarrow \Lambda \cap D^\prime(L)_p(\pi(\nu))$.
\item For every node $\nu \prec \mu \prec v$, if there is an edge \begin{tikzpicture}[baseline = (lambda.base),node distance=1.5cm, auto]
\node (lambda) {$\mu$};
\node (lambda_prime) [right of = lambda] {$v$};
\draw[-] (lambda) to node {$q$} (lambda_prime);
\end{tikzpicture} in $G(J)$.
\begin{itemize}
\item Find the unique node $u$ completing the square \begin{tikzpicture}[baseline =(anchor.base), node distance=1.5 cm, auto]
\node (lambda1) {$u$};
\node (lambda2) [right of = lambda1] {$\mu$};
\node (lambda3) [below of = lambda1] {$\nu$};
\node (lambda4) [below of = lambda2] {$v$};
\draw[-] (lambda1) to node {$p$} (lambda2);
\draw[-] (lambda1) to node (anchor) [left]  {$q$} (lambda3);
\draw[-] (lambda3) to node [below] {$p$} (lambda4);
\draw[-] (lambda2) to node {$q$} (lambda4);
\end{tikzpicture}, with $u\prec \nu$ and $d(u)<d(v)$.
\item $(\star)$ \label{square_verif}
$\Lambda = \Lambda \cap \text{the missing piece of}$ \begin{tikzpicture}[baseline =(anchor.base), node distance=1.5 cm, auto]
\node (lambda1) {$\pi(u)$};
\node (lambda2) [right of = lambda1] {$\pi(\mu)$};
\node (lambda3) [below of = lambda1] {$\pi(\nu)$};
\node (lambda4) [below of = lambda2] {$?$};
\draw[-] (lambda1) to node {$p$} (lambda2);
\draw[-] (lambda1) to node (anchor) [left]  {$q$} (lambda3);
\draw[-] (lambda3) to node [below] {$p$} (lambda4);
\draw[-] (lambda2) to node {$q$} (lambda4);
\end{tikzpicture} by \Cref{lemma_square}.
\end{itemize}
\end{enumerate}
			\item  For every $\lambda\in\Lambda$, $\pi(v)\leftarrow\lambda$ and
			\begin{enumerate}
			\item If $v$ is the biggest element of depth $d$ then $d\leftarrow d+1$ and $v\leftarrow$ the smallest element of depth $d+1$ and go to (1).
			\item Else $v\leftarrow$ the next element in lexicographical order of depth $d$ and go to (1).
			\end{enumerate}
            \end{enumerate}
        \end{enumerate}
\end{itemize}
\end{algorithm}
\begin{proof}[Proof of the algorithm]
If we denote by $G(J)_{\prec v}$ the graph $G(J)$ restricted to the nodes smaller than $v$, then the loop invariant is $A(v)$=\emph{``The image by the puzzle $\pi$ of every node $\nu\prec v$ is defined and every subsquare of $G(J)_{\prec v}$ is sent by $\pi$ to a realizable square"}.
The fact that $A((1,\ldots,1))$ is verified is trivial.
Let us suppose that for some $v\succ(1,\ldots,1)$, $A(v)$ is verified. And let us achieve one more step in the loop. Let $v^\prime$ be the successor of $v$ for $\prec$. The subsquares $G(J)_{\prec v^\prime}$ are all realizable since they either do not contain $v^\prime$ and thus be realizable since $A(v)$ is verified or contains $v^\prime$ and are realizable thanks to $(\star)$ in Procedure (2).

The algorithm stops when it reaches the node $v_\infty$ of highest depth and the greatest for the lexicographic order. Then, every nodes of $G(J)$ is colored by the puzzle $\pi$ and all subsquares of $G(J)_{\preceq v_\infty} = G(J)$ are realizable, which means the obtained puzzle is realizable, as desired.
\end{proof}

\section{Complexity comparisons}\label{complexity}

Let us compare the performances of the three methods we have to obtain every characteristic maps over a wedged seed PL sphere $L(J)$.

\begin{enumerate}
\item The Garrison and Scott algorithm directly applied on $L(J)$;
\item \Cref{old_puzz} applied on $(L,J)$;
\item \Cref{new_puzz} applied on $(L,J)$.
\end{enumerate}

In all that follows, the elementary operations we consider for the algorithm are the summation of $\Z_2$-vectors (this operation is done in $\cO(1)$ since we can use a binary encoding for such vectors), and the access to an element of an array.

First of all, let us show that \Cref{new_puzz} is way better than \Cref{old_puzz}.

\subsection{Puzzle algorithm, old versus new}\label{old_vs_new_puzz}

Let us first describe more explicitly the complexity of \Cref{old_puzz}.

This algorithm computes the following.
\begin{itemize}
\item In the initialization, we compute all the DJ classes of $L$, $\wed_p(L),p\in [m]$ and $\wed_{p,q}(L), p,q\in[m],p\neq q$ with the Garrison and Scott algorithm.
\item For each of the $(\prod_{k=1}^m j_k)^{|\CM(L)|}$ possible puzzles,
\begin{itemize}
\item In procedure (2), we check the $\sum_{k=1}^{m}(\binom{j_k-1}{2}\prod_{l\neq k}j_l)$ edges of the puzzle.
\item In procedure (3), we check the $\sum_{k=1}^{m}(\binom{j_k-1}{2}\binom{j_k-2}{1}(\prod_{l\neq k}j_l-1)) = \sum_{k=1}^{m}((j_k-2)\binom{j_k-1}{2}(\prod_{l\neq k}j_l-1))$ squares of the puzzle.
\end{itemize}
\end{itemize}

Then it leads to the following proposition.

\begin{proposition}
The complexity of \Cref{old_puzz} is
$$\cO\left(m^2\times GS(\wed_{p,q}(L)) + \left((\prod_{k=1}^m j_k)^{|\CM(L)|}\right)\times\left( \sum_{k=1}^{m}(j_k-2)\binom{j_k-1}{2}\left(\prod_{l\neq k}j_l-1\right)\right)\right),$$
with $GS(K)$ being the complexity of the Garrison and Scott algorithm on a given simplicial complex $K$.

\end{proposition}

On the contrary, \Cref{new_puzz} computes the following.

\begin{itemize}
\item To provide the input data, we only need to compute $\CM(L)$.
\item Let us give an upper bound for the number of possible realizable puzzles. For each node $v$ of the graph $G(J)$, we will associate an image $\pi(v)$ which is determined from its neighbours of smaller or equal depth. However \Cref{CP_nice_prop} stays that once the nodes of depth $\leq 1$ have their images determined, then the puzzle is either entirely determined or will possess some square which is not realizable and thus will not be realizable. As a consequence, for depth $> 1$, the algorithm will simply verify that all the squares are indeed realizable for the remaining nodes. We call the depth $\leq 1$ coloring of the graph $G(J)$ a \emph{determining corner} of the graph. Thus, for finding an upper bound on the number of realizable puzzles, we only need to count the  number of determining corners. The algorithm will associate an image $\lambda^0$ to the node $(1,\ldots,1)$ among the $|\CM(L)|$ choices possible.
    Then the algorithm will choose one DJ class $p$-adjacent to $\lambda^0$ for every nodes $w$ which are $p$-adjacent to $v$ (of depth $1$, and for every $p$) and there are $(\len(J)-1)$ such $w$. After looking at \Cref{constr_puzzle}, one can notice that for any DJ class $\lambda^0$, there is a maximum of $2^{\Pic(L)}$ DJ classes which can be $p$-adjacent to $\lambda^0$ since two $p$-adjacent edges differ by a $\psi\in\Z_2^{\Pic(L)}$. Thus, the number of realizable puzzles is bounded by $|\CM(L)|\times(2^{\Pic(L)})^{\len(J)}$.
\item For each remaining node $\nu$, we will check if its image by $\pi$ can be the missing piece of a square. To do so, we check the number of preceding pair of edges $\nu$ is part of. The number of such edges is bounded by $\len(J)$ and we take two of them. Thus this whole process has a complexity bounded by $(|\V(G(J))|)^{\len(J)^2} =\left(\prod_{k=1}^m j_k\right)^{\len(J)^2}$. This is checked for every determining corners.
\end{itemize}

This thus leads to the following proposition.

\begin{proposition}
The complexity of \Cref{new_puzz} is
$$\cO\left(\left(\prod_{k=1}^m j_k\right)^{\len(J)^2}\times |\CM(L)|\times 2^{\Pic(L)\len(J)}\right).$$
\end{proposition}

\begin{remark}
In practice, the number of DJ classes over $L$ tends to be very small when the dimension increases for a fixed Picard number.
Also notice that we provided a very rough upper bound on the number of characteristic maps which are $p$-adjacent for some vertex $p$. In fact, the structure of the prediagram of $L$ influences a lot on the time complexity. Namely, many cases of determining corners will be discarded in no time. In addition, when the components of $J$ increase all together, this gives more restrictions when computing the depth $2$ and the number of determining corners providing a realizable square will then drastically decrease.
\end{remark}

One can see that when $\len(J)$ becomes bigger, then the complexity of \Cref{old_puzz} becomes drastically worse than the one of \Cref{new_puzz}.

\subsection{The performance of the \Cref{new_puzz} versus a direct use of the Garrison and Scott algorithm}

Let $J$ be an $m$-tuple of positive integers with $j:=\len(J)$. A DJ class of $L(J)$ (the ones which are searched with the GS algorithm) is written as $$\lambda = \begin{blockarray}{ccccc}
1\ldots n+j & n+j+1 & n+j+2 & \ldots & m+j \\
\begin{block}{(c|cccc)}
I_{n+j} & \lambda_{n+j+1} & \lambda_{n+j+2} & \ldots & \lambda_{m+j}\\
\end{block}
\end{blockarray}.$$

In order to apply \Cref{new_puzz} we need two steps. The first one is to compute $\CM(L)$ by the Garrison and Scott algorithm on the seed $L$. If $j>0$, since this algorithm is applied on $L$, it will for sure be more efficient than the Garrison and Scott algorithm on $L(J)$. The second step is the main part of the algorithm and depends only on $j$, on $p$ the Picard number of $L$, and on the number of elements in $\CM(L)$.

Let us compute the  minimal complexity for the Garrison and Scott algorithm applied on a pure simplicial complex $K$ on $[m]$, with $\dim(K)=n-1$. First, recall that the main idea of the Garrison and Scott algorithm is to progressively enumerate all the possible images $\lambda(n+l)\in \Z_2^n$ for $l=1,\ldots,m-n$ by removing the sums of the form $\sum_{i\in\sigma}\lambda(i)$ with $ \sigma \cup \{n+l\}\in K\cap[n+l]$ from the set $\Z_2^n$.
The branch-and-bound tree has a depth equal to $m-n$. Let us say that, for a parent node $P$ of depth $l$, this removal operation deletes $N\geq0$ images.

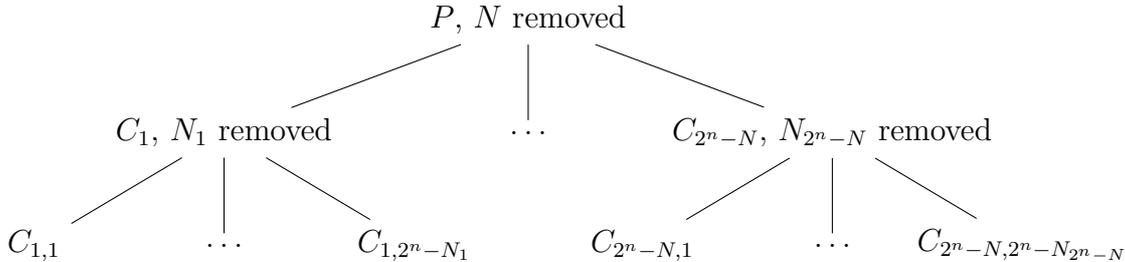
\begin{tikzpicture}[
    level 1/.style = {sibling distance = 4cm},
    level 2/.style = {sibling distance = 2.5cm}
]
\node {$P$, $N$ removed}
    child  {node {$C_1$, $N_1$ removed}
    child  {node {$C_{1,1}$} }
    child {node {$\ldots$}}
    child  {node {$C_{1,2^n-N_1}$}} edge from parent}
    child {node {$\ldots$}}
    child  {node {$C_{2^n-N}$, $N_{2^n-N}$ removed}
    child  {node {$C_{2^n-N,1}$}}
    child {node {$\ldots$}}
    child  {node {$C_{2^n-N,2^n-N_{2^n-N}}$}} edge from parent};
\end{tikzpicture}

Since removing one element is an elementary operation, this makes the complexity to be $\cO(N)$ for this step. The node $P$ will have $2^n-N$ children, say $C_1,\ldots,C_{2^n-N}$. For each child $C_k$, we will have to remove the sums of the form $\sum_{i\in\sigma}\lambda(i)$, with $\sigma \cup \{n+l+1\}\in L\cap[n+l+1]$, say $N_k$ distinct sums from $\Z_2^n$. Thus, increasing the number of child increases the time complexity exponentially.

 So the thinner the tree is at step $l$, the smaller the time complexity will be. The smallest tree for the Garrison and Scott algorithm has its nodes having only one son at each step, except for the last one (multiplicity equal to $|\CM(L(J))|$). The depth of the tree will of course be equal to $p=m-n$ and the number of operations done for each node is the number of colors we remove so it is $2^{n}$. The total complexity is then equal to $\cO\left(p\times2^{n}\right)$.

As an example, let us compute the explicit amount of sums $\sum_{i\in\sigma}\lambda(i)$ such that $\sigma \cup \{n+1\}\in L\cap[n+1]$, namely, the first step in the Garrison and Scott algorithm.
We consider $K$ a PL sphere.
It should satisfy the \emph{pseudo manifold} condition which is that any $(n-2)$-face should be included in exactly two $(n-1)$-faces.
Let us suppose that the $(n-2)$-face $f=\{1,\ldots,n-1\} \subset F_0=\{1,\ldots,n\} (\in K)$ is also included in $\{1,\ldots,n-1,n+1\}$ this means that every subset of $f$ is a face of $K$ and this gives us $2^{n-1}$ sums which are removed from the $2^n$ first available choices for $\lambda(n+1)$, since $\lambda(i)=e_i$, $i=1,\ldots,n$. Of course some other sums may be removed by the other faces (of dimension $\leq n-1$) which contain $n+1$.

This computation also assures us that since a $(n-2)$-face $f$ containing the vertex $n+l$, $l=1\ldots,p$ must be in some $(n-1)$-face $F$, then this face will contain at least $n-p$ vertices from the face $F_0$, which leads to removing at least $2^{n-p}$ sums at each step. Once more, the Picard number $p$ is fixed and this increases exponentially with respect to the dimension $n-1$.

Now if we come back to our actual case, we study the complexity of the Garrison and Scott algorithm for $K=L(J)$ having dimension $n+j-1$. Thus the lower bound for the complexity of the Garrison and Scott algorithm on $L(J)$ is $\cO\left(p\times2^{n+j}\right)$.

This latter complexity depends exponentially on $n$. Then for a fixed number of wedge operations $j$ and Picard number $p$, if we increase $n$, then at some point the puzzle algorithm will be faster than the Garrison and Scott algorithm.

Here is a benchmark obtained at Picard number $4$ in which we can observe this result with $|J|=3$.

\begin{table}[h]
\begin{tabular}{|c|c|c|c|c|c|c|}
\hline
$(n,m)$& (2,6) & (3,7)& (4,8) & (5,9) & (6,10) & (7,11)\\
Number of PL spheres & 1 & 4 & 21 & 142 & 733 & 1190 \\ \hline
Garrison and Scott (time for one) & 0.05s & 0.85s & 0.78s & 1.93s & 6s & 112s \\ \hline
\Cref{new_puzz} (time for one) & 0.35s & 0.31s & 0.11s & 0.02s & 0.013s & 0.010s \\ \hline
\end{tabular}
\caption{Comparison between the Garrison and Scott algorithm and \Cref{new_puzz}.}
\end{table}

Note that for obtaining the previous table, we used a variation of the Garrison and Scott algorithm which in practice makes the computation of $\CM(L)$ faster when $L$ is a seed and the Picard number $p$ is small. A description of this algorithm can be found in \Cref{IDCM_GS}.

\section{Conclusion and discussion}

The main improvement of \Cref{new_puzz} is that the computation of mod 2 characteristic maps on  $L(J)$ with a fixed Picard number $p$ now only depends on
\begin{itemize}
\item The complexity of the Garrison and Scott algorithm on the seed $L$ which is of smaller dimension and has fewer vertices, and on
\item The number of wedge operations $j$ performed on $L$ and the size of $\CM(L)$.
\end{itemize}

This algorithm can also determine quite quickly if a determining corner does not lead to a realizable puzzle (if a square does not have a missing piece or if two missing pieces coming from two different squares provide a different coloring for the same node then the algorithm stops completing this puzzle).

For toric topological purposes, one would like to extend this algorithm for $\Z$ characteristic maps. However, it is known that missing pieces for a square are not unique in this case (see \cite{CP_wedge_2}) and this refrains us from using the same algorithm. However when one works on a complex toric variety $X$, it is one-to-one associated with a fan which in some cases is associated to a characteristic map $\lambda$ on a simplicial complex $K$. Then, the real projection $X^\R$ of $X$ will lead to the (mod 2) fan which is in the same cases associated to the (mod 2) reduction $\lambda^\R$ of $\lambda$ over the same $K$. So $X^\R$ will support (at least) one characteristic map $\lambda^\R$.
Conversely, if $K$ supports a (mod 2) characteristic map $\lambda^\R$ then this simplicial complex is a good candidate for being associated with fan and then to a complex toric variety. Furthermore, if $K$ does not support any (mod 2) characteristic map then it cannot be associated to a complex toric variety.
Hence, it is enough to consider the seeds $L$ which support (mod 2) characteristic maps and use the puzzle algorithm to construct the small covers associated with them. In fact, if a seed does not support any (mod 2) characteristic map, then its wedges will also not support any as a consequence of the puzzle method.

For the case of wedge of polygons in \cite{Choi-Park2019}, the authors show that determining corners can lead for sure to a realizable puzzle by providing some additional rules on the determining corners. This shows that with some modification on \Cref{new_puzz} we can stop it at depth $1$. One could then wonder if this is true for a general case and try to find some rules for more general cases.

\newpage
\begin{appendix}

\section{The IDCM Garrison and Scott algorithm}\label{IDCM_GS}

One interesting fact about \Cref{new_puzz} is that it only requires to compute the Garrison and Scott algorithm on $L$ which is a seed PL sphere. Seed PL spheres which are not direct product of simplices have their dual characteristic maps $\bar{\lambda}$ which all are \emph{injective}, namely, $\bar{\lambda}(i)\neq\bar{\lambda}(j), i\neq j$. This fact can be used for creating a variation of the Garrison and Scott algorithm.

This section describes a different version of the Garrison and Scott algorithm. Since this algorithm is computed on IDCM (Injective Dual Characteristic Maps), the colors for the vertices of $L$ will be in $\Z_2^p$, so one can ``intuitively" think that for a ``small" Picard number and an $n$ ``big enough", this algorithm will be faster than the Garrison and Scott algorithm. In this section we only give some rough idea of the complexities without proper proofs. Only a the benchmark table at the end allows to convince ourselves that our intuition seems correct.

Let $L$ be a seed PL sphere on $[m]$ with $\dim(L)=n-1$.

As said before, any characteristic map on $L$ will have its dual injective, so we can use the following variation of the Garrison and Scott algorithm for finding $\CM(L)$, we will call it the IDCM Garrison and Scott algorithm.

\begin{algorithm}[A modification of Algorithm~4.1 in \cite{Garrison-Scott2003}]
    \begin{itemize}
    \item[ ]
    \item \textbf{Input:} $CF$ = the union of the power sets of all cofacets of $L$.
    \item \textbf{Output:} $\Gamma$ = the list of $\Z_2$ vectors $( \lambda_1, \ldots , \lambda_m )$ such that the last $m-n$ vectors form the standard basis $\Z_2^{m-n}$.
    \item \textbf{Initialization:}
        \begin{itemize}
            \item[] $\lambda_{n+1} \leftarrow (1,0, \ldots , 0), \lambda_{n+2} \leftarrow (0, 1, \cdots, 0), \ldots, \lambda_m \leftarrow (0,0, \ldots, 1)$.
            \item[] $\Gamma \leftarrow \emptyset $
            \item[] $S \leftarrow $ the list of nonzero elements of $\Z_2^{m-n}$
            \item[] $i \leftarrow n$
        \end{itemize}
    \item \textbf{Procedure:}
        \begin{enumerate}
            \item\label{dualproc1} Set $S_i \leftarrow S \setminus\{\lambda_j,j>i\}$.
            \item For all $I \in CF$ of the form $I = \{i\} \cup \{ i_1, \ldots, i_k \}$ with $1 \leq i \leq i_1 \leq \cdots \leq i_k $, remove the vector $\lambda_{i_1} + \cdots + \lambda_{i_k} $ from the list $S_i$.
            \item If $i=n$, then STOP.
            \item If $S_i = \emptyset$, then $ i \leftarrow i+1$ and go to (3).
            \item Set $\lambda_i \leftarrow$ the 1st element of $S_i$ and remove it from $S_i$
            \item If $i=1$, then add $ ( \lambda_1, \ldots , \lambda_m ) $ to $\Gamma$ and go to (3).
            \item Set $i \leftarrow i-1$ and go to (1).
        \end{enumerate}
\end{itemize}
\end{algorithm}

Let us explain why this algorithm is intuitively more efficient than the classic GS algorithm for small Picard numbers.

For a simplicial complex $K$ on $[m]$, and $i\in[m]$, we define $K_{\leq i} :=\{f\cup{i}\subset[i]\mid f\in K\}$, the set of faces of $K$ having all their vertices in the set $[i]$. The \emph{star} of $K$ at a face $F\in K$ is denoted as $\St_K(F) := \{G\in K\mid F\subset G\}$, and is the set of all faces of $K$ containing the face $F$. We denote by $\co(K)$ the PL sphere whose maximal faces are the complementary of the maximal faces of $K$, we then have $\dim(\co(K)) = \Pic(K)-1$.

The parameters which impact the complexity of the branch-and-bound algorithm are the following:
\begin{itemize}
\item The size of the research tree;
\item The size of the sets we use for deleting branches of the tree.
\item The complexity of the basic operations we use (here we deal with sum of vectors in $\Z_2^a$ for some $a$).
\end{itemize}

On one hand, the GS algorithm will work on the complete tree $\cT_{\text{GS}}$ whose nodes are every possible vectors $\lambda_{n+k}\in\Z_2^{n},k=1,\ldots,\Pic(L)$ such that $\left(\begin{array}{c|ccc}
1\ldots n & n+1 & \ldots & m \\ \hline
I_{n+j} & \lambda_{n+1} & \ldots & \lambda_{n+i}\end{array}\right)$ is a characteristic map of $L_{\leq i}$. We have $|\cT_{GS}| = (2^{n}-1)^{\Pic(L)}$. The removal of a branch of the tree is made by finding the elements of $\St_{L_{\leq i}}(i)$ for $i=n+1,\ldots,m$ in $L$ to make the wanted characteristic map respect the non-singularity condition. If we suppose that we calculated $\St_{L_{\leq i}}(i)$ for $i=1,\ldots,\Pic(L)$ at the very beginning, we then have to compute the $|\St_{L_{\leq i}}(i)|\geq 2^{(n-1)}-(n-1)$ or $=0$ linear combinations of vectors in $\Z_2^{n}$ since there is either at least one maximal face $F$ containing $i$ and then each proper sub-face of $F$ leads to one linear combination or no such maximal faces (but in this case less branches are deleted).

On the other hand, the IDCM Garrison and Scott algorithm works on a tree which is of depth $n$ and the number of children starts from $2^{\Pic(L)}-1$ and is decremented at each step. The tree $\cT_{\text{IDCM GS}}$ is then of  size $|\cT_{\text{IDCM GS}}| = \frac{(2^{\Pic(L)}-1-\Pic(L))!}{(2^{\Pic(L)}-1-\Pic(L)-n)!}$. Moreover, the removal of branches is done by finding the faces in $co(L)_{\geq i}$ and in this case $|co(L)_{\geq i}| \geq  2^{(\Pic(L)-1)}-(\Pic(L)-1)$. Furthermore the linear combinations are computed in $\Z_2^{\Pic(L)}$.

Here is a table of the given complexities for Picard number 4.

\begin{table}[h]
\begin{tabular}{|c|c|c|c|c|c|c|c|c|c|c|} \hline
$(n,m)$ & (2,6) & (3,7) & (4,8) & (5,9) & (6,10) & (7,11) & (8,12) & (9,13) & (10,14) & (11,15) \\ \hline
$|\cT_{\text{GS}}|$ & 81& 2e+3&5e4&9e+05&2e+7&3e+8&4e+9&7e+10&1e+12&1e+13 \\ \hline
$\St_{L_{\leq i}}(i)|\geq$ & 2&5&12&27&58&121&248&503&1e+3&2e+3\\ \hline \hline
$|\cT_{\text{IDCM GS}}|$ & 110&990&8e+3&6e+4&3e+5&2e+6&7e+6&2e+7&4e+7&4e+7\\ \hline
$|co(L)_{\geq i}| \geq$& 11&11&11&11&11&11&11&11&11&11\\ \hline
\end{tabular}
\caption{Comparative table of the complexities.}
\end{table}

For Picard number 4 and $n>4$, the tree is smaller with the IDCM version of the algorithm and there are fewer linear combinations to compute, which also are done in a smaller vector space. However we cannot tell precisely about the number of branches which will be removed. We just know that the branch removal operations are faster with the IDCM version.

The following table gives some benchmark on the time efficiency of the IDCM algorithm if compared to the Garrison and Scott version.

\begin{table}[h]
\begin{tabular}{|c|c|c|c|c|c|c|}
\hline
$(n,m)$& (2,6) & (3,7)& (4,8) & (5,9) & (6,10) & (7,11)\\
Number of PLS & 1 & 4 & 21 & 142 & 733 & 1190 \\ \hline
G-S (time for all)& 0.2ms & 3ms & 50ms & 1.5s & 25s & 353s \\ \hline
IDCM G-S (time for all)& 0.4ms & 7ms & 82ms & 1.7s & 14s & 15s \\ \hline
\end{tabular}
\caption{Time efficiency comparison between the Garrison and Scott algorithm and the IDCM modification.}
\end{table}

\end{appendix}
\bibliographystyle{amsplain}
\providecommand{\bysame}{\leavevmode\hbox to3em{\hrulefill}\thinspace}
\providecommand{\MR}{\relax\ifhmode\unskip\space\fi MR }
\providecommand{\MRhref}[2]{%
  \href{http://www.ams.org/mathscinet-getitem?mr=#1}{#2}
}
\providecommand{\href}[2]{#2}

\end{document}